\documentclass[11pt,a4paper,reqno]{amsart}

\calclayout
\usepackage{thmtools}
\usepackage[hidelinks]{hyperref}
\usepackage{etoolbox}
\usepackage[capitalize,noabbrev]{cleveref}
\usepackage{amsmath}
\usepackage{amsfonts}
\usepackage{amsthm}
\usepackage{amssymb}
\usepackage{mathtools}
\usepackage{tikz-cd}
\usepackage{stmaryrd}
\usepackage{mathabx}
\usepackage[shortlabels]{enumitem}
\usepackage{mathrsfs}
\usepackage{url}
\usepackage{tkz-graph}
\usepackage{standalone}

\usepackage[alphabetic]{amsrefs}

\newtheorem{Thm}{Theorem}[section]
\newtheorem{Lm}[Thm]{Lemma}
\newtheorem{Cor}[Thm]{Corollary}

\newtheorem{Prop}[Thm]{Proposition}

\newtheorem{IntroThm}{Theorem}

\theoremstyle{definition}
\newtheorem{Ex}[Thm]{Example}

\newtheorem{Qst}[Thm]{Question}

\newtheorem{Df}[Thm]{Definition}

\theoremstyle{remark}
\newtheorem{Rmk}[Thm]{Remark}

\newcommand{\Fc}{\ensuremath{\mathcal{F}}}

\newcommand{\abs}[1]{\ensuremath{\left\lvert #1 \right\rvert}}

\newcommand{\inner}[1]{\ensuremath{\left\langle #1 \right\rangle}}
\newcommand{\ext}[1]{\ensuremath{{\bigwedge}_{\QQ}\inner{#1}}}

\DeclareMathOperator{\del}{\backslash}

\DeclareMathOperator{\ann}{ann}
\DeclareMathOperator{\deq}{\vcentcolon=}
\DeclareMathOperator{\cl}{cl}
\DeclareMathOperator{\rk}{rk}

\DeclareMathOperator{\OS}{\mathsf{OS}}
\DeclareMathOperator{\rOS}{\overline{\mathsf{OS}}}
\DeclareMathOperator{\OT}{\mathsf{OT}}
\DeclareMathOperator{\AOT}{\mathsf{AOT}}
\newcommand{\GMA}[1]{\mathsf{GMA}(#1)}

\DeclareMathOperator{\A}{\mathcal{A}}
\newcommand{\CC}{\mathbb{C}}
\newcommand{\QQ}{\mathbb{Q}}
\newcommand{\ZZ}{\mathbb{Z}}
\newcommand{\RR}{\mathbb{R}}
\newcommand{\KK}{\mathbb{K}}
\DeclareMathOperator{\HS}{\mathrm{HS}}

\newcommand{\Cc}{\mathcal{C}}
\DeclareMathOperator{\si}{si}
\DeclareMathOperator{\codim}{codim}
\DeclareMathOperator{\tr}{Tr}

\title{Koszul Orlik--Solomon Algebras from Non-supersolvable Arrangements}

\author{Tuong Le}
\address{Princeton University, Department of Mathematics, Princeton, NJ, USA}
\email{tuongle@princeton.edu}

\author{Chayim Lowen}
\address{Princeton University, Department of Mathematics, Princeton, NJ, USA}
\email{chayiml@princeton.edu}

\author{Jason McCullough}
\address{Iowa State University, Department of Mathematics, Ames, IA, USA}
\email{jmccullo@iastate.edu}

\date{\today}

\begin{document}

\begin{abstract}
The cohomology ring of the complement of a  complex hyperplane arrangement is given by its Orlik--Solomon algebra.  It is known that the defining ideal of the Orlik--Solomon algebra has a quadratic Gr\"obner basis in the standard presentation if and only if the intersection lattice is supersolvable; such algebras are automatically Koszul.  In 1997, Shelton and Yuzvinsky posed the question as to whether all Koszul Orlik--Solomon algebras arise from supersolvable arrangements.  We answer this question negatively using three related constructions that produce non-supersolvable arrangements whose Orlik--Solomon algebras are Koszul.  Moreover, these arrangements may be chosen to be irreducible, realizable over $\mathbb{Q}$, and of any rank $\geq 3$. 
Our constructions rely on a result of Falk and Proudfoot which we strengthen and generalize.
In two of the three cases, we show non-supersolvability using a corrected form of a result of Ziegler regarding supersolvability of parallel connections. We also construct Koszul Orlik--Terao algebras coming from non-supersolvable arrangements.
\end{abstract}



\maketitle

\section{Introduction}

The study of the topology and combinatorics of complex hyperplane arrangements has a rich history.  
One highlight is the presentation of the singular cohomology ring of the complement of a complex hyperplane arrangement $\mathcal{A}$ by Orlik and Solomon in \cite{OrlikSolomon}, now called the \textit{Orlik--Solomon algebra of $\A$} and denoted by $\OS(\mathcal{A})$. 
Of particular note is that $\OS(\mathcal{A})$ depends only on the intersection lattice $\mathcal{L}(\A)$ of $\A$.
There have since been considerable attempts to identify properties of the intersection lattice of a hyperplane arrangement that imply, or are equivalent to, ``good'' algebraic consequences for the arrangement. One such property is supersolvability.
Supersolvable lattices are those with a full flag of modular elements; see \Cref{subsec:matroids} for a definition in our setting.
These were introduced by Stanley \cite{supersolvable}, who showed that the characteristic polynomial of a supersolvable lattice splits into linear factors over $\mathbb{Z}$.

It was later shown by Bj\"orner--Ziegler \cite{BZ91}*{Theorem 2.8} and Peeva \cite{Peeva}*{Theorem 4.3} that if $\A$ is a central and essential complex hyperplane arrangement, then $\mathcal{L}(\A)$ is supersolvable if and only if the defining ideal $I(\A)$ of the Orlik--Solomon algebra of $\A$  has a quadratic Gr\"obner basis with respect to the standard generating set and some monomial order.
By a result of Terao \cite{Terao86}*{Corollary 2.17}, this is equivalent to the arrangement being fiber-type. 
It follows that the Orlik--Solomon algebra of a supersolvable arrangement is Koszul in the sense of Priddy \cite{Priddy}.
Whether the converse holds was first raised by Shelton and Yuzvinsky in 1997.
\begin{Qst}[{\cite{SY97}*{p.\ 487}}]\label{qst:mainquestion}
    If the Orlik--Solomon algebra $\OS(\A)$ of an arrangement $\A$ is a Koszul algebra, must the intersection lattice $\mathcal{L}(\A)$ be supersolvable?\footnote{An earlier paper \cite{FR86} claimed that the equivalent statement ``rational $K(\pi,1)$ complement implies supersolvability of the intersection lattice'' was false, but this was based on an incorrect claim by Kohno at the time \cite{FR00}*{p.\ 102}.  The equivalence between the Koszul property for $\OS(\A)$ and the rational $K(\pi,1)$ property for $\mathcal{C}(\A)$ was not known until later  \cite{PY99}.}
\end{Qst}
\noindent This question was raised---sometimes as a conjecture in the affirmative---several times in the literature, including \cite{PY99}*{p.\ 165}, \cite{FR00}*{Problem 2.2}, \cite{Yuzvinskii01}*{Problem 6.7}, \cite{Schenck12}*{Problem 82}, \cite{LR13}*{Conjecture 1.3}, \cite{HW16}*{p.\ 187}, \cite{Coron25}*{Conjecture 6.4}, and \cite{PS09}*{Problem 16.7}---where it was called ``a very difficult open question.''  An affirmative answer is known  for several special classes of arrangements:  Jambu and Papadima \cite{JP02} proved this for hypersolvable arrangements, Schenck and Suciu  \cite{SS02} for graphic arrangements,  Van Le and R\"omer  \cite{LR13} for arrangements with pairwise disjoint minimal broken circuits, Hultman \cite{Hultman16} for root ideal arrangements, and Lutz \cite{Lutz19}  for Dirichlet arrangements.
Our main result is a negative answer to \Cref{qst:mainquestion}.

\begin{IntroThm}\label{thm:full_counterexample}
    For any $r \ge 3$, there is an irreducible, central hyperplane arrangement of rank $r$ defined over $\QQ$ with Koszul Orlik--Solomon algebra but non-supersolvable intersection lattice.
    Moreover, the arrangement can be chosen round (see \Cref{subsec:matroids} for a definition).
\end{IntroThm}
\noindent For a rank 3 arrangement with these properties, see \Cref{ex:rank_3_Koszul_nonSS}.  

\begin{Rmk}
    The answer to \Cref{qst:mainquestion} was expected to be negative by many experts in the field. Various potential counterexamples had been proposed; see e.g.\ \cite{SY97}*{Example 5.4} and \cite{FR86}*{p.\ 108}.
    The difficulty however is that the Koszul property is known \emph{not} to be checkable on any finite portion of the free resolution of the coefficient field \cite{Roos93}.
    Any proof of this property must therefore proceed by finding a ``good reason'' for Koszulness. This is precisely the role of our Theorems \ref{thm:main} and \ref{thm:GPC_Koszul} below.
\end{Rmk}

For a complex arrangement $\A$, Papadima and Yuzvinsky \cite{PY99} showed that $\OS(\A)$ is Koszul if and only if the complement $\mathcal{C}(\A)$ is a rational $K(\pi,1)$ space.
The latter means that the $\QQ$-completion of $\mathcal{C}(\A)$ in the sense of Bousfield and Kan \cite{BK72} is aspherical.
Using topological methods, Falk and Proudfoot \cite{falk_proudfoot} showed that this property for $\A$ can be inferred from the same property for two simpler arrangements: 
\begin{Thm}[{\cite{falk_proudfoot}*{Corollary 3.6}}]\label{thm:FP}
    Let $\A$ be a central arrangement in $\mathbb{C}^{n}$ and let $X$ be a modular flat of $\mathcal{L}(\A)$. The arrangement $\A$ has a rational 
    $K(\pi, 1)$ complement provided the same is true of the localization $\mathcal{A}_X$ and of the arrangement $\mathcal{A}_{\pi}$, where $\A_{\pi}$ is the arrangement induced by $\A$ on a generic fiber of the quotient map $\pi: \mathbb{C}^r \to \mathbb{C}^r/X$.
\end{Thm} 
\noindent
The intersection lattice $\mathcal{L}(\A_{\pi})$ depends only on $\mathcal{L}(\A)$. By \cite{falk_proudfoot}*{Theorem 2.1}, it is the \emph{complete principal truncation} of $\mathcal{L}(\A)$ along $X$;  see \Cref{df:complete_trunc}.
The importance of this construction in this context was recognized earlier by Brylawski \cite{constructions} in relation to the aforementioned theorem of Stanley on the splitting of the characteristic polynomial.
We prove a result which generalizes and strengthens \Cref{thm:FP}.
We state our result in the language of matroids.
Given a matroid $M$ and a modular flat $F$ of $M$, we write $\overline{T}_F(M)$ for the complete principal truncation $M$ along $F$. 
\begin{IntroThm}\label{thm:main}
    Let $M$ be a matroid and $F$ a modular flat of $M$ of positive rank.
    The Orlik--Solomon algebra of $M$ is Koszul if and only if the same is true of both $M|_F$ and $\overline{T}_F(M)$.
\end{IntroThm}
\noindent
Since we cannot avail ourselves of the topological techniques in \cite{falk_proudfoot}, we instead rely crucially on a result of Terao \cite{Terao86}*{Theorem 3.8} on the decomposition of the Orlik--Solomon algebra induced by a modular flat.
Remarkably, if we replace the Koszul property for the Orlik--Solomon algebra with the supersolvability of the matroid then \emph{neither} implication in \Cref{thm:main} holds.
This gives two \emph{distinct} mechanisms to produce examples of non-supersolvable matroids with Koszul Orlik--Solomon algebras. We also provide a third mechanism that combines both failures.

Exploiting the failure of the forward implication, we show in Sections \ref{sec:first_counterexample} and \ref{sec:modular_no_flag} how to construct non-supersolvable complete principle truncations of supersolvable matroids, giving a first mechanism.
We note that, incidentally, complete principal truncations were likewise used by Whittle \cite{Whittle} to disprove a conjecture of Welsh \cite{Welsh80}*{Conjecture 8.4}.

Exploiting the failure of the backward implication, we show in Sections \ref{sec:modular_joins} and \ref{sec:parallel_connections} how to construct non-supersolvable parallel connections of supersolvable matroids, which we show always have Koszul Orlik--Solomon algebras. This is our second mechanism.

Combining the ideas of both, in \Cref{sec:GPC_Koszul} we show:
\begin{IntroThm}\label{thm:GPC_Koszul}
    A generalized parallel connection of two matroids with Koszul Orlik--Solomon algebras has a Koszul Orlik--Solomon algebra.
\end{IntroThm}
\noindent
As our third mechanism, we construct non-supersolvable generalized parallel connections with Koszul Orlik--Solomon algebras.
In particular, such examples can be found of any rank $\geq 5$ in the class of irreducible signed-graphic arrangements; see Examples \ref{ex:pcnonss} and \ref{ex:non_ss_modular_join}, as well as the discussion preceding the latter.

So far, we have explained how \Cref{thm:main} can be used to construct arrangements with Koszul Orlik--Solomon algebras.
However, we would also like to be sure that our arrangements are not supersolvable. 
As a practical matter, this property can be checked by a finite computation in a reasonable time for any given arrangement (unlike the Koszul property).
However, we explain general reasons why supersolvability fails in our constructions.
In the case of the first mechanism, the reason is given in \Cref{cor:truncation_not_ss}.
In the two other cases, non-supersolvability is guaranteed by the following \emph{corrected form} of a theorem of Ziegler.
\begin{IntroThm}\label{thm:ss:modular:join}
Let $M$ be a matroid and $F$ an inclusion-minimal modular flat in $M$ such that 
$M$ is a non-trivial modular join over $F$.
Write $M = P_{F}(M|_{F'}, M|_{F''})$ for some flats $F', F''$ of $M$ containing $F$.
Then $M$ is supersolvable if and only if $M|_{F'}$ and $M|_{F''}$ are both supersolvable and one of them has a full modular flag passing through $F$.
\end{IntroThm}
\noindent
In \cite{Z91}, the statement mistakenly requires that $F$ belong to modular flag in \textit{both} $M|_{F'}$ and $M|_{F''}$.  \Cref{rmk:error} shows this not to be correct.

For any arrangement $\A$ there are commutative analogues of the Orlik--Solomon algebra known as the Orlik--Terao algebra $\OT(\A)$ and the Artinian Orlik--Terao algebra $\AOT(\A)$.
Orlik and Terao \cite{OrlikTeraoAlgebras} used the Artinian Orlik--Terao algebra to count chambers in real hyperplane arrangements, settling a conjecture of Aomoto \cite{Aomoto96}.
Schenck and Tohaneanu used the quadratic component of the defining ideal of the (non-Artinian) Orlik--Terao algebra to characterize $2$-formal arrangements \cite{SchenckTohaneanuOTAlgebras}.
It was known that if $\mathcal{L}(\A)$ is supersolvable, then both $\OT(\A)$ and $\AOT(\A)$ have quadratic Gr\"obner bases and thus are Koszul.
Whether the converse holds was raised in \cite{Denham14}*{Question 4.7}.  
Inspired by our use of parallel connections (our second mechanism) to answer \Cref{qst:mainquestion}, we show  that the same construction answers this question negatively as well.
\begin{IntroThm}\label{thm:koszul:nonss:OT}
    For any rank $r \ge 5$, there is an irreducible, central complex hyperplane arrangement $\A$ such that both $\OT(\A)$ and $\AOT(\A)$ are Koszul but $\mathcal{L}(\A)$ is not supersolvable.
\end{IntroThm}
\noindent

\subsection*{Roadmap}
The rest of this paper is organized as follows.
\Cref{sec:background} contains the necessary background material on matroids, arrangements, and Koszul algebras.
In \Cref{sec:first_counterexample}, we prove \Cref{thm:main} and use it to construct a first counterexample to the Shelton-Yuzvinsky problem.
In \Cref{sec:modular_no_flag} we explain how to generalize the example in the previous section using Dowling geometries and prove \Cref{thm:full_counterexample}.
\Cref{sec:modular_joins} studies when the modular join of two matroids is supersolvable, leading to a correct form of Ziegler's result in \Cref{thm:ss:modular:join}.
In \Cref{sec:parallel_connections}, we apply this knowledge to produce non-supersolvable parallel connections with Koszul Orlik--Solomon algebras.
We also answer a question of Falk and Randell regarding the existence of combinatorially distinct arrangements with homotopy equivalent complements.
In \Cref{sec:koszul:OT}, we construct Koszul Orlik--Terao algebras from non-supersolvable arrangements, proving \Cref{thm:koszul:nonss:OT}.  In \Cref{sec:GPC_Koszul}, we show that generalized parallel connections preserve Koszulness of the Orlik--Solomon algebra, proving \Cref{thm:GPC_Koszul}.
In the final \Cref{sec:questions}, we ask a number of natural questions about the relationship between the OS-Koszul property and other well-known algebraic properties of hyperplane arrangements.

\section{Background}\label{sec:background}
\subsection{Matroids}\label{subsec:matroids}
We briefly review some aspects of the theory of matroids and lay out conventions that will be used in this text. 
For a thorough treatment of matroids, see \cite{Oxley} or \cite{white}.

Given a matroid $M$, we write $E(M)$ for its \textbf{ground set}, $\rk M$ for its rank, and $\rk_M S$ for the rank of any subset $S \subseteq E$. We say that $M$ is \textbf{trivial} if $E(M) = \varnothing$.
We write $\mathcal{L}(M)$ for the \textbf{lattice of flats} of $M$ and $\cl_M$ for the \textbf{closure} operator of $M$.
The closure of a subset of $E(M)$ is its \textbf{span}.
A subset $J \subseteq E(M)$ is \textbf{spanning} if its span is all of $E(M)$.
It is \textbf{dependent} if $J$ is contained in the span of some proper subset $J' \subsetneq J$ and \textbf{independent} otherwise.
A \textbf{circuit} is an inclusion-minimal dependent set. A \textbf{basis} is a maximal independent set; equivalently, a minimal spanning set.
Given flats $F$ and $G$ of $M$, we write $F \vee G$ for their \textbf{join}, i.e.\  the flat $\cl_M(F \cup G)$. For a subset $S$ of $M$, we write
$M/S$ for the \textbf{contraction} of $M$ by $S$ and $M\del S$ for the \textbf{deletion} of $S$ from $M$. We write $M|_S$ for the \textbf{restriction}\footnote{A warning: this nomenclature clashes with the notion of {restriction} for hyperplane arrangements, which is instead related to matroid contraction.} of $M$ to $S$, which coincides with the deletion $M\del (E - S)$.
A matroid obtained from $M$ by restriction is a \textbf{submatroid} of $M$. If this restriction is to a spanning subset, we call it a \textbf{spanning submatroid} of $M$.

\textit{In this paper, matroids will always be assumed \emph{\textbf{loopless}}:} 
the rank of any singleton must be 1. 
They will sometimes additionally be assumed \textbf{simple}: the rank of each pair is 2. 
Any loopless matroid $M$ can be made simple by identifying all elements in each atom.
The matroid thus obtained from $M$ is its \textbf{simplification}, denoted $\si(M)$. 
The map $M \mapsto \mathcal{L}(M)$ induces a bijection between simple matroids and \textbf{geometric lattices} (each considered up to isomorphism).
The rank 1 flats of a matroid are called \textbf{atoms} or \textbf{points}.
In simple matroids, these may be identified with the elements of the ground set.
The rank 2 flats are called \textbf{lines}.
The maximal proper flats of a matroid are its \textbf{coatoms}.\footnote{These are sometimes called \emph{hyperplanes}, a convention we avoid here for obvious reasons.} 
A \textbf{coloop} of a matroid $M$ is an element $x \in E(M)$ whose \emph{complement} $E(M) - \{x\}$ is a coatom.

The \textbf{direct sum} of two matroids $M$ and $N$ is the matroid $M \oplus N$ whose ground set is the disjoint union $E(M) \sqcup E(N)$ and whose flats are the unions $F \cup G$ for $F \in \mathcal{L}(M)$, $G \in \mathcal{L}(N)$. 
A matroid is \textbf{disconnected} if can be expressed as a direct sum of non-trivial matroids;
equivalently, if its ground set may be written as the (necessarily disjoint) union of two proper flats of complementary rank.
A matroid is \textbf{connected} if it is not disconnected. 
Any matroid may be expressed in an essentially unique way as a direct sum of (possibly many) connected matroids.
A property stronger than connectedness (for loopless matroids) is \textbf{roundness}.
A matroid is round if its ground set cannot be written as the union of two proper flats whatsoever.
From the definition, we get:
\begin{Prop}\label{prop:round_contract}
    Let $M$ be a round matroid and $F$ a flat of $M$. Then $M/F$ is round.
\end{Prop}
As an example, if $G$ is a finite simple connected graph, its graphic matroid $M(G)$ (see \Cref{subsec:graphic}) is round if and only if $G$ is a complete graph \cite{Oxley}*{p. 326}.
The following is a useful sufficient criterion for roundness.
\begin{Prop}\label{prop:round_criterion}
    Let $M$ be a simple matroid.
    Suppose we can find a basis $B$ of $M$ such that each two-element subset of $B$ lies in a 3-element circuit of $M$. 
    Then $M$ is round.
\end{Prop}
    \begin{proof}
        Suppose instead that the ground set of $M$ were the union of two proper flats $F$ and $G$. 
        Neither $F$ nor $G$ is spanning in $M$, so neither contains all of $B$. Choose (necessarily distinct)  $b \in B - F$ and $b' \in B - G$. Let $x$ be any element of $M$ such that $C = \{b,b',x\}$ is a 3-element circuit. 
        We have $b \in G - F$ and $b' \in F - G$. Without loss of generality, $x \in F$.
        Then $b$ is the unique element of the circuit $C$ not contained in the flat $F$.
        This is impossible.
    \end{proof}
\noindent
We now recall the basic properties of modularity and supersolvability in matroids.
\begin{Df}[{\cite{constructions}*{Theorem 3.3}}]
\label{df:modular}
    The following are equivalent for a flat $F$ in a matroid $M$:
    \begin{enumerate}[(i)]
    \item For any flat $G$ in $L$, we have
    \(
        \rk_M F + \rk_M G = \rk_M (F \vee G) + \rk_M (F \cap G).
    \)\label{it:eq_modular}
    \item For any flat $G$ of $M$ with $\rk_M F + \rk_M G > \rk M$, we have $\rk_M (F \cap G) > 0$.\label{it:complement}
    \end{enumerate}
    We say $F$ is \textbf{{modular}} in $M$ if it satisfies these equivalent conditions.
\end{Df}
\noindent
Flats of $M$ of rank $0$, $1$ or $\rk M$ are always modular.
Modular flats have a number of desirable properties. Here are some that we will need.
\begin{Prop}\label{prop:basic}
    Let $M$ be a matroid and $F, G \in \mathcal{L}(M)$.
    \begin{enumerate}
        \item Suppose $F \subseteq G$,
        $F$ is modular in $M|_G$, and $G$ is modular in $M$. Then $F$ is modular in $M$.\label{it:transitivity}
        \item If $F$ is modular in $M$, then $F \cap G$ is modular in $M|_G$.\label{it:restriction}
        \item If $F, G$ are modular in $M$, then $F \cap G$ is modular in $M$.\label{it:intersection}
        \item If $F$ is modular in $M$, then $(F \vee G)-G$ is modular in $M/G$.\label{it:quotient}
        \item If $F$ is modular in $M$ and $S \subseteq E(M)$ contains $F$, then $F$ is modular in $M|_S$.\label{it:deletion}
        \item If $F$ is modular in $M$ and $M$ is connected, then $M|_F$ is connected.\label{it:connected}
    \end{enumerate}
\end{Prop}
\begin{proof}
    \eqref{it:transitivity} is \cite{constructions}*{Proposition 3.5}.
    \eqref{it:restriction} is \cite{modular}*{Lemma 2}.
    \eqref{it:intersection} is \cite{constructions}*{Proposition 3.6}.
    \eqref{it:quotient} is a specialization of \cite{constructions}*{Corollary 3.9}.
    \eqref{it:deletion} is \cite{constructions}*{Proposition 3.8}.
    \eqref{it:connected} is the second half of \cite{constructions}*{Corollary 3.16}.
\end{proof}
\noindent
The following result, due to Brylawski, relates modular flats and components of a matroids.
\begin{Prop}[{\cite{constructions}*{Proposition 3.19}}]\label{prop:modular_descent}
    Let $M$ be a matroid and $F \in \mathcal{L}(M)$.
    Suppose $M/F$ is disconnected.
    Choosing $G, H \in \mathcal{L}(M)$ such that $G \cap H = F$ and $G \cup H = E(M)$, we have that $G$ is modular in $M$ if and only if $F$ is modular in $M|_H$.
\end{Prop}
\noindent
We now define supersolvability of matroids.
\begin{Df}
    Let $M$ be a matroid of rank $r$.
    A \textbf{full modular flag} in $M$ is a sequence of modular flats $F_0 \subsetneq F_1 \subsetneq \dots \subsetneq F_r$ of $M$ such that $\rk F_i = i$ for all $i$.
    If such a sequence exists, $M$ is said to be \textbf{supersolvable}.
\end{Df}
\noindent
Importantly for us, supersolvability holds for a matroid $M$ if and only if it holds for its simplification, since it is a property of the lattice of flats.
We will say that a full modular flag \textbf{passes through} a flat $F$ if this flat is part of the sequence defining it.
The following is immediate from the definition of supersolvability and the transitivity of modularity (\Cref{prop:basic}\eqref{it:transitivity}).
\begin{Lm}\label{lm:coatom_supersolvable}
    Let $M$ be a matroid. Then $M$ is supersolvable if and only if it admits a coatom $H$ such that $H$ is modular in $M$ and $M|_H$ is supersolvable.
\end{Lm}
\noindent
Supersolvability is heritable in the following sense.
\begin{Prop}[{\cite{supersolvable}*{Proposition 3.2(i)}}]\label{prop:inherit_ss}
    Let $M$ be a supersolvable matroid and let $F \subseteq G$ be flats of $M$. Then $M|_G/F = M/F|_{G-F}$ is supersolvable.
\end{Prop}
\noindent
We end this subsection with the all-important notion of complete principal truncation of modular flats.
We remark that complete principal truncations may be derived from the simpler concept of principal truncations and that both of these make sense for any flat, not necessarily modular.
We refer the reader to \cite{Brylawski86}*{\S~7.4} for a discussion of these constructions.
Since we will only need the special case of modular flats, we will use the following theorem of Brylawski (which requires modularity) as our \emph{definition} of the complete principal truncation.
\begin{Df}[{\cite{constructions}*{Proposition 5.14(3)}}]\label{df:complete_trunc}
    Let $M$ be a matroid and $F$ a modular flat of $M$ of positive rank.
    The \textbf{complete principal truncation} $\overline{T}_F(M)$ is the matroid on $E(M)$ whose flats are precisely those flats of $M$ which either contain $F$ or are disjoint from it.
\end{Df}
\noindent
Note that $\overline{T}_F(M)$ will in general not be simple, even when $M$ is simple.
It is however loopless provided $M$ is loopless.
The following describes the rank function of a complete principal truncation (in our restricted sense).
    \begin{Lm}\label{lm:rank_function}
        Let $M$ be a matroid and let $F$ be a modular flat of $M$ of positive rank.
        The rank of a flat $G$ of $\overline{T}_F(M)$ is given by
        \[
            \rk_{\overline{T}_F(M)} G = \begin{cases}
                \rk_M G & \text{if $F \cap G = \varnothing$}\\
                \rk_M G - \rk F +1& \text{if $F \subseteq G$}
            \end{cases}
        \]
        In particular, $\rk \overline{T}_F(M) = \rk M - \rk_M F + 1$.
    \end{Lm}
\begin{proof}
    In the first case, let $r = \rk_M G$ and let $\varnothing = G_0 \subsetneq \dots \subsetneq G_r = G$ be a saturated chain of flats in $M$.
    This same chain is saturated in $\overline{T}_F(M)$, showing that $G$ has rank $r$ in this matroid too.

    In the second case, let $r = \rk_M G$.
    Let $s = \rk_M F$.
    Let $F = G_0 \subsetneq G_1 \subsetneq \dots \subsetneq G_{r-s} = G$ be a saturated flag of flats in $M$ linking $F$ to $G$.
    Then $\varnothing \subsetneq F \subsetneq G_1  \subsetneq \dots \subsetneq G_{r-s} = G$ is a saturated chain of flats in $\overline{T}_F(M)$, showing that $G$ has rank $r-s+1$ in this matroid.
\end{proof}
\noindent
Consideration of the flats of $\overline{T}_F(M)$ and their ranks shows:
\begin{Prop}\label{prop:truncation_round}
    Let $M$ be a matroid and $F$ a modular flat of $M$. If $M$ is connected (resp.\ round), then $\overline{T}_F(M)$ is connected (resp.\ round).
\end{Prop}
\begin{proof}
    Suppose $G, H \in \mathcal{L}(\overline{T}_F(M))$ satisfy $G \cup H = E(\overline{T}_F(M)) = E(M)$.
    Then $G, H \in \mathcal{L}(M)$, so if $M$ is round then $G = E(M)$ or $F = E(M)$.
    Suppose $M$ is merely connected but that $\rk_{\overline{T}_F(M)}G + \rk_{\overline{T}_F(M)}H = \rk {\overline{T}_F(M)}$.
    This implies $G \cap H = \varnothing$.
    So exactly one of $G$ or $H$, say $G$, contains $F$. By \Cref{lm:rank_function},
        \[
            \rk_M G + \rk_M H = 
            (\rk_{\overline{T}_F(M)}G + \rk F - 1) + \rk_{\overline{T}_F(M)} H = \rk \overline{T}_F(M) + \rk F - 1 = \rk M.
        \]
        The connectedness of $M$ then implies that $G = E(M)$ or $H = E(M)$.
\end{proof}
    
\subsection{Hyperplane arrangements}
A complex\footnote{One may consider arrangements over other fields, but we will not do so here.} hyperplane arrangement $\mathcal{A}$ is a finite collection of complex affine hyperplanes in $\CC^r$. See \cite{ArrangementsOT} for an introduction to the subject.
Two arrangements $\A$ and $\mathcal{B}$ in $\CC^r$ are \textbf{isomorphic} if the hyperplanes of $\A$ are carried bijectively onto those of $\mathcal{B}$ by a bijective affine-linear map $\CC^r \to \CC^r$.
An arrangement is \textbf{trivial} if it contains no hyperplanes.
The \textbf{center} of an arrangement $\A$ is the intersection $\bigcap_{H \in \A} H$.
An arrangement $\A$ is \textbf{central} if it contains $0$ in its center (i.e.\ if its hyperplanes are linear subspaces). It is \textbf{essential} if the intersection of some subcollection of its hyperplanes consists of a single point.  
Thus $\A$ is both central and essential precisely when its center is $\{0\}$.
Write $x_1, \dots, x_r$ for the dual standard basis of $\CC^r$, and let $S = \CC[x_1,\ldots,x_r]$ be the coordinate ring.
Then we can write each $H \in \A$ as the \textbf{vanishing set} $V(\ell_H)$ of some linear polynomial $\ell_H \in S$. For central arrangements, these are homogeneous linear forms.
By abuse of notation, we will often identify $H$ with $\ell_H$ and $\A$ with the set 
$\{\ell_H \mid H \in \A\}$.
We say that $\A$ is \textbf{defined over} $\QQ$ (resp.\ $\RR$) if the polynomials $\ell_H$ may be chosen to have coefficients in $\QQ$ (resp.\ $\RR$).
The union $\bigcup_{H\in \A} H$
is the \textbf{variety} of $\A$.
The \textbf{complement} of $\A$ is $\mathcal{C}(\A) = \CC^r - \bigcup_{H \in \A} H$, the complement of its variety in $\CC^r$. 

Suppose $\A$ is a central arrangement in $\CC^r$ and $H_0 \in \A$ is a hyperplane, where $H_0 = V(\ell_0)$. The \textbf{decone} of $\A$ with respect to $H_0$ is the arrangement $\mathbf{d}_{H_0} \A = \{H \cap H_0' \mid H \in \A, \; H \neq H_0\}$, where $H_0'$ is the \emph{translated} hyperplane $\{\ell_0 = 1\}$.
This arrangement naturally lives in $H_0'$, which one identifies with $\CC^{r-1}$ via an affine-linear isomorphism.
Every point in $\Cc(\A)$ may be expressed uniquely as a scalar multiple of a point in $H_0'$.
This gives a diffeomorphism
$\Cc(\A) \simeq \Cc(\mathbf{d}_{H_0}\A) \times \mathbb{C}^{\times}$.

A \textbf{flat} of $\mathcal{A}$ is an intersection of some set of its hyperplanes. We allow also the empty intersection, which is identified with the ambient space $\CC^r$. 
The flats of a central arrangement $\A$, ordered by reverse inclusion and ranked by codimension, form a \textbf{geometric lattice} $\mathcal{L}(\A)$, called the \textbf{intersection lattice} of $\A$.\footnote{When $\A$ is not central, $\mathcal{L}(\A)$ is merely a \textit{geometric semilattice}, as defined in \cite{Wachs}.}  The bottom element is $\hat{0} = \CC^r$, the atoms are the hyperplanes in $\A$, and the top element is $\hat{1} = \bigcap_{H \in \A} H$ (which is $\{0\}$ if $\A$ is essential). Equivalently, the hyperplanes in $\A$ form the ground set of a simple matroid $M(\A)$. A flat $X$ of $\A$ is then identified\footnote{This identification \emph{reverses} containment! We will rely on context to disambiguate the meaning of $\subseteq$.} with the flat $\{H \in \A \mid H \supseteq X\}$ of $M(\A)$. This is the point of view we will adopt for most of this paper.  The \textbf{rank} of $\A$, denoted $\rk \A$, is the maximum number of linearly independent hyperplanes in $\A$; when $\A$ is central, $\rk \A = \rk M(\A)$.  
We say that a flat of $\A$ is \textbf{modular} if it is modular in $M(\A)$.
We say that $\A$ is \textbf{supersolvable} if 
$M(\A)$ is. For any arrangement $\A$, we say that $\A$ \textbf{realizes} $M(\A)$.
A matroid isomorphic to $M(\A)$ for some central arrangement $\A$ is said to be \textbf{realizable over $\CC$}.
If the arrangement may be chosen to be defined over $\QQ$ (resp.\ $\RR$), then the matroid is \textbf{realizable over $\QQ$} (resp.\ \textbf{realizable over $\RR$}). 
Most matroids are not realizable over $\CC$
\cite{Nelson}.
Given a central arrangement $\A$ and a hyperplane $H_0 \in \A$, the flats of $\mathbf{d}_{H_0}\A$ are identified with the flats of $\A$ that are not contained in $H_0$.

The \textbf{M\"obius function} of a central arrangement $\mathcal{L}(\A)$ is the function $\mu: \mathcal{L}(\A) \to \mathbb{Z}$ defined recursively by
\[
\mu(F) = \begin{cases} 1 & \text{if $F = \hat{0}$}\\
-\sum_{G \supsetneq F} \mu(G) & \text{otherwise}.
\end{cases}
\]
The \textbf{Poincar\'e polynomial of $\A$} is the polynomial $\pi(\A,t) = \sum_{F \in \mathcal{L}(\A)} \mu(F) \cdot (-t)^{\rk F}$.

Let $\A$ be an arrangement in $\CC^r$.  If $\mathcal{B} \subseteq \A$ is a subset, then $\mathcal{B}$ is a \textbf{subarrangement}.
In this case, $M(\mathcal{B})$ is a submatroid of $M(\A)$.
For a flat $X \in \mathcal{L}(\A)$, the \textbf{localization} $\mathcal{A}_X$ of $\A$ at $X$, is
\[
    \A_X = \{H \in \A \mid H \supseteq X\}.
\]
The matroid $M(\A_X)$ is the \emph{restriction} 
$M(\A)|_X$ of $M(\A)$ to $X$.
Confusingly, \textbf{restriction} of $\A$ to $X$ is the name for the arrangement $\A^X$ in $X$ given by
\[
    \A^X = \{X \cap H \mid H \in \A - \A_X \text{ and } X \cap H \neq \varnothing\}.
\]
The matroid $M(\A^X)$ is the \emph{contraction}
$M(\A)/X$ of $M(\A)$ by $X$.

The \textbf{product} of hyperplane arrangements $\mathcal{A}$ in $\mathbb{C}^{r}$ and $\mathcal{B}$ in $\mathbb{C}^{s}$ is the arrangement $\mathcal{A} \times \mathcal{B}$ in $\mathbb{C}^{r+s}$ consisting of hyperplanes $H \times \mathbb{C}^{s}$ for $H \in \mathcal{A}$ and $\mathbb{C}^{r} \times H$ for $H \in \mathcal{B}$.
An arrangement which cannot be expressed as a product of two nontrivial arrangements is \textbf{irreducible}.\footnote{In \cite{Arrangements}, this term is reserved for a slightly different notion; our notion is there called \emph{indecomposability}.} The arrangement $\A$ is irreducible if and only if the its matroid $M(\A)$ is connected. Similarly, the matroid $M(\A \times \mathcal{B})$ is the direct sum of the matroids $M(\A)$ and $M(\mathcal{B})$.
Any arrangement can be expressed in an essentially unique way as a product of irreducible ones.

\subsection{Signed-graphic arrangements}\label{subsec:graphic}
We will make extensive use of \textbf{signed-graphic} arrangements and of the underlying signed-graphic matroids. The latter were introduced by Zaslavsky in \cite{Zaslavsky82}.
A \textbf{multigraph} is a graph with possibly both loops and parallel edges.
A \textbf{signed graph} is a multigraph in which each non-loop edge has been labeled with one of the symbols $+$ or $-$.

Suppose we are given a signed graph $\Gamma$ on vertices $v_1, \dots, v_r$. The signed-graphic arrangement associated to this graph 
is the central arrangement in $\CC^r$ with the following hyperplanes:
\begin{itemize}
    \item The hyperplane $\{x_i = 0\}$ whenever $\Gamma$ has a loop at the vertex $v_i$,
    \item The hyperplane $\{x_i = x_j\}$ whenever $\Gamma$ has an edge labeled $+$ joining $v_i$ and $v_j$ with $i < j$,
    \item The hyperplane $\{x_i = -x_j\}$
    whenever $\Gamma$ has an edge labeled $-$ joining $v_i$ and $v_j$ with $i < j$.
\end{itemize}
The signed-graphic matroid associated to a signed graph is the underlying matroid of the arrangement, except that each element of the matroid occurs with the same multiplicity as the corresponding edge type (i.e.\ this is the size of its parallel class).
By construction, signed-graphic matroids are $\QQ$-realizable.  We call the special case where there are no loops and no edges labeled $-$ the \textbf{graphic arrangement} of the associated (unsigned) graph $G$; the corresponding matroid $M(G)$ is the \textbf{graphic matroid} of $G$.

The \textbf{complete signed-graphic} arrangement (resp.\ matroid) of order $r$ is the signed-graphic arrangement (resp.\ matroid) associated to a signed graph on $r$ vertices in which all possible edge types are present exactly once. 
Signed graphic arrangements are precisely the subarrangements of complete signed graphic arrangements.
The latter form a subclass of the class of \textbf{Dowling geometries}, which will be discussed in \Cref{sec:modular_no_flag}.
In a signed-graphic matroid associated to a signed graph $G$, the edges corresponding to any vertex-induced submultigraph of $G$ form a flat. 

\subsection{Free resolutions and Koszul algebras}
Let $\KK$ be a field, and let $A$ be a standard graded $\KK$-algebra.\footnote{Algebras are always assumed unital and associative; they need not be commutative.}
Thus we have a $\KK$-vector space decomposition
$A = \bigoplus_{i = 0}^\infty A_i$ and $A_i \cdot A_j \subseteq A_{i+j}$ for all $i,j \ge 0$.
The adjective \textbf{standard}
here means that $A_0 = \mathbb{K}$ (i.e.\ $A$ is \textbf{connected}) and $A$ is generated as a $\mathbb{K}$-algebra by a finite number of elements from $A_1$.
This implies that each graded piece $A_i$ is finite-dimensional.
Also $A_+ \deq \bigoplus_{i = 1}^\infty A_i$ is a two-sided ideal of $A$, thereby equipping $\KK = A/A_+$ with the structure of an $(A,A)$-bimodule. 
Linear maps between graded $A$-modules are always assumed to \emph{preserve} the grading.
Given a graded left $A$-module $M$ and an integer $j \in \ZZ$, the \textbf{$j$-shifted} module $M(-j)$ is the graded $\mathbb{K}$-vector space with components given by $M(-j)_i = M_{i-j}$, equipped with its natural structure as a graded left $A$-module.
A graded left $A$-module is \textbf{free} if it is isomorphic to a direct sum of shifted copies of $A$.
Let $M = \bigoplus_{j \in \ZZ} M_i$ be a finitely-generated graded left $A$-module. 
Then $M_i$ is finite-dimensional for all $i \in \ZZ$ and $M_i = 0$ for $i \ll 0$.  The $A$-module $M$ admits a \textbf{minimal free resolution} over $A$, i.e.\ there is a long exact sequence of graded left $A$-modules
\[
\begin{tikzcd}
0 & M\ar[l] & F_0\ar[l] & F_1\ar[l] & F_2\ar[l] & \cdots\ar[l]
\end{tikzcd}
\]
in which the $F_i$ are free and
the induced maps $\KK \otimes_A F_i \leftarrow \KK \otimes_A F_{i+1}$ are zero for all $i \ge 0$.  Minimal free resolutions are unique up to isomorphism.  
A resolution as above is \textbf{linear} if $F_i = A(-i)^{\oplus \beta_i}$ for all $i \ge 0$ and some nonnegative integers $\beta_i$.

A standard graded $\KK$-algebra $A$ is \textbf{Koszul} if the graded left module $A/A_+ \cong \KK$ has a linear free resolution over $A$.
If $A$ is Koszul then it is \textbf{quadratic}: it can be defined as the quotient of a tensor algebra (or exterior algebra, or polynomial ring) by an ideal of quadratic forms.  If its defining ideal has a quadratic Gr\"obner basis, then $A$ must be Koszul.
Neither of the above implications is reversible in general.
The following criteria will allow us to transfer Koszulness between algebras. 
\begin{Thm}[{\cite{QAlg05}*{\S~3.2 Corollary 5.4}}]\label[Thm]{thm:regularity_Koszul}
    Suppose $R$ and $S$ are  standard graded $\mathbb{K}$-algebras and we are given a surjective graded $\mathbb{K}$-algebra homomorphism $R \twoheadrightarrow S$. Assume that $S$ has a linear free resolution as a left $R$-module.
    Then $R$ is Koszul if and only if $S$ is Koszul.
\end{Thm}

\begin{Thm}[{\cite{HHO18}*{Theorem 2.31}}]\label[Thm]{thm:Koszul_retracts}
    Let $R, S$ be  standard graded $\mathbb{K}$-algebras. Suppose that
    $S$ is a graded algebra retract of $B$, i.e.\ there are graded algebra homomorphisms 
    $S \hookrightarrow R$ and $R \twoheadrightarrow S$ which compose to the identity map $S \to S$. 
    If $R$ is Koszul then $S$ is Koszul.
\end{Thm}

\begin{Thm}[{\cite{BF85}*{Theorem 4(c)}}] \label[Thm]{thm:tensor_product_Koszul}
    Let $R, S$ be standard graded-commutative  $\mathbb{K}$-algebras.
    The graded-commutative tensor product algebra $R \otimes_{\mathbb{K}} S$ is Koszul if and only if both $R$ and $S$ are.
\end{Thm}

\noindent
For a  standard graded $\mathbb{K}$-algebra $A$, we consider the generating function
 $P^A_{\KK}(t) = \sum_{i = 0}^\infty \beta_i^A t^i$
 of the ranks of the free modules in the minimal free resolution of $\KK$ as an $A$-module. It is the \textbf{Poincar\'e series} of $A$. 
 If $A$ is Koszul, its Poincar\'e series satisfies $P^A_{\KK}(t)\HS_A(-t) = 1$, where $\HS_A(t) = \sum_{i \ge 0} \dim_{\KK} (A_i) t^i$ is the \textbf{Hilbert series} of $A$. If $A$ is quadratic, it has a \textbf{quadratic dual} $A^!$, and $A$ is Koszul if and only if $A^!$ is Koszul.  When $A$ is Koszul, $HS_A(t) = P_{\KK}^{A^!}(t)$ and $HS_{A^!}(t) = P_{\KK}^A(t)$.  
For a quick introduction to Koszul algebras, we refer the reader to \cite{Froberg99}. A more thorough treatement can be found in \cite{QAlg05}.

\subsection{Orlik--Solomon algebras}\label[subsection]{subsec:OS_algebras}
Let $\A = \{H_1,\ldots,H_n\}$ be a central arrangement in $\CC^r$.
Let $\Lambda(\A) := \ext{e_1,\ldots,e_n}$ 
be the exterior algebra of $\QQ^n$.
We identify the standard basis of $\QQ^n$ with the hyperplanes in $\A$ and write $e_i$ and $e_{H_i}$ interchangeably.
The \textbf{Orlik--Solomon ideal} of $\A$ is the homogeneous ideal $I(\A) \subseteq \Lambda(\A)$ generated by $\{\partial e_C \mid C \text{ is a circuit of $M(\A)$}\}$. Here, if $C = \{H_{i_1}, \ldots, H_{i_t}\}$ with $i_1 < \dots < i_t$ then
\[
    e_C \deq e_{i_1} \cdots e_{i_t}, \quad \text{and} \quad \partial e_C \deq \sum_{j = 1}^t (-1)^{j-1} e_{i_1}\cdots \widehat{e_{i_j}} \cdots e_{i_t}.
\]
The \textbf{Orlik--Solomon algebra} of $\A$ is the quotient ring $\OS(\A) := \Lambda(\A)/I(\A)$. By construction it is graded-commutative. Though this definition appears to depend on a choice of total order on $\A$, there is in fact no such dependence since $\partial e_C$ depends on the total order on $C$ only up to a sign.
The work of Brieskorn \cite{briesk73} and Orlik--Solomon \cite{OrlikSolomon} shows that the singular cohomology ring $H^*(\Cc(\A), \QQ)$ is isomorphic to $\OS(\A)$.\footnote{Though Orlik and Solomon work with complex coefficients, the isomorphism holds with integral coefficients by Brieskorn's Lemma \cite{briesk73}*{Lemma 5}, hence with coefficients in any field; see \cites{ArrangementsOT, Yuzvinskii01}. For simplicity, we stick to using rational coefficients for the Orlik--Solomon algebra. Note that questions of Koszulness may depend  on the characteristic of the coefficient field (see e.g.\ \cite{Dali23}*{Proposition 5.15}). In fact, our arguments and examples work in any characteristic.}
For any central arrangement, we have $\HS_{\OS(\A)}(t) = \pi(\A,t)$.
A consequence of the above presentation is that this cohomology ring depends as a graded $\QQ$-algebra only on the intersection lattice of $\A$, not on the defining linear forms or their coefficients. 

In fact, the above definition makes sense verbatim for any geometric lattice, hence for any simple matroid. 
We write $\OS(M)$ for the Orlik--Solomon algebra of a simple matroid $M$ so defined and $e_x$ for the generator corresponding to an element $x \in E(M)$.
As a matter of fact, the definition makes sense if $M$ is merely assumed loopless. In this case, $e_x$ and $e_y$ are identified in $\OS(M)$ whenever $x$ and $y$ are parallel, inducing a graded $\QQ$-algebra isomorphism $\OS(M) \cong \OS(\si(M))$.
By construction, $\OS(\A) = \OS(M(\A))$ for any central arrangement $\A$.

When $\A$ is not (necessarily) central, we define $\OS(\A)$ similarly by instead setting the ideal $I(\A)$ to be generated by $\{\partial e_C\mid C \text{ is a circuit of $\A$}\} \cup \{e_D \mid \bigcap_{H \in D} H = \varnothing\}$.
This again gives a presentation for the singular cohomology ring $H^*(\Cc(\A), \QQ)$. 
For a hyperplane $H_0$ of a central arrangement $\A$, comparing the presentations for $\OS(\A)$ and 
$\OS(\mathbf{d}_{H_0}\A)$, one concludes that $\OS(\mathbf{d}_{H_0}\A) = \OS(\A)/(e_{H_0})$.
For a matroid $M$ and an element $x \in E(M)$,
though we do not assign meaning to the symbol $\mathbf{d}_xM$, we \emph{define} the 
$\mathcal{L}(\mathbf{d}_x M)$ to be the subposet of $\mathcal{L}(M)$ consisting of flats that do not contain $x$; we similarly define
$\OS(\mathbf{d}_x M)$ to be the algebra 
$\OS(M)/(e_x)$. An alternative presentation for $\OS(\mathbf{d}_x M)$ is as the quotient of the free exterior algebra on $E(M) - \{x\}$ by the homogeneous ideal generated by 
\[
    \{\partial e_C\mid C \text{ is a circuit of $M \del x$}\} \cup \{e_D \mid x \in \cl_M(D)\},
\]
more closely resembling the definition we first gave above in the case of arrangements.

We return to the setting of a central arrangement $\A = \{H_1, \dots, H_n\}$.
The ring $\OS(\mathbf{d}_{H_0}\A)$ is \emph{independent} of the choice of $H_0 \in \A$, up to isomorphism. This is best seen by defining the \textbf{reduced Orlik--Solomon algebra} $\rOS(\A)$ of $\A$ to be the $\QQ$-subalgebra generated by the differences 
$e_{H} - e_{H'}$ for $H, H' \in \A$. 
Starting with the observation that, for any $1 \leq i \leq n$,
\[
    \Lambda(\A) = \ext{e_1,\ldots,e_n} = \ext{e_i} \otimes_{\QQ}
    \ext{e_1-e_2, e_2 - e_3,\ldots,e_{n-1}-e_n}
\]
as a tensor product of graded-commutative $\QQ$-algebras, and that for a circuit $C = \{H_{i_1}, \ldots, H_{i_t}\}$ with $i_1 < \dots < i_t$ the element
\[
\partial e_C = (e_{i_2} - e_{i_1})(e_{i_3} - e_{i_1}) \cdots (e_{i_t} - e_{i_1})
\]
lives in the second factor, we deduce
$\OS(\A) = \ext{e_i} \otimes_{\QQ} \rOS(\A)$ and
$\rOS(\A) \cong \OS(\A)/(e_i) = \OS(\mathbf{d}_{H_i} \A)$. 
Similar observations were made in \cite{EPY03}*{Proposition 1.3}.
This logic extends to the case of a matroid $M$, for which we can similarly define $\rOS(M)$ and obtain 
$\rOS(M) \cong \OS(\mathbf{d}_{x} M)$ for all $x \in E(M)$. We note that the tensor factorization above implies $\HS_{\OS(M)}(t) = (1+t) \HS_{\rOS(M)}(t)$.

\begin{Prop}\label[Prop]{prop:OSKoszul}
    Let $M$ be a matroid and $x \in E(M)$ an element. The Koszul condition for $\OS(M)$, for $\OS(\mathbf{d}_{x}M)$, and for $\rOS(M)$ are equivalent.\footnote{This fact is certainly known to experts but we were not able to find a reference for it.}
\end{Prop}
\begin{proof}
The latter two are isomorphic, so it suffices to show equivalence for the first and third algebras.
The isomorphism $\OS(M) \cong \ext{e_{H_0}} \otimes_{\QQ} \rOS(M)$ shows by \Cref{thm:tensor_product_Koszul} that 
$\OS(M)$ is Koszul if and only if 
$\ext{e_{H_0}}$ and $\rOS(\A)$ both are. The conclusion now follows from \Cref{thm:tensor_product_Koszul} and the well-known fact that free exterior algebras are Koszul.
\end{proof}
\noindent
A matroid is \textbf{OS-Koszul} if it satisfies the equivalent conditions in \Cref{prop:OSKoszul}.
An arrangement $\A$ is \textbf{OS-Koszul} if $M(\A)$ is OS-Koszul. Setting $A = \Lambda(\A)/(I(\A)_2)$ to be the quadratic Orlik--Solomon algebra 
 of $\A$, the quadratic dual $A^!$ is the universal enveloping algebra $U(\mathfrak{h}_{\A})$  of the holonomy Lie algebra of $\A$ \cite{Schenck12}*{p.\ 342}.  If $\OS(\A)$ is Koszul, so that 
$A = \OS(\A)$, one has 
\[\pi(\A,-t) = \HS_{\OS(\A)}(-t) = \frac{1}{P^{\OS(\A)}_{\QQ}(t)} = \frac{1}{\HS_{U(\mathfrak{h}_{\A})}(t)} = \prod_{i = 1}^\infty (1-t^i)^{\phi_i},
 \] where $\phi_i$ are the LCS ranks of $\pi_1(\mathcal{C}(\A))$; this is
 the Lower Central Series (LCS) formula in \cite{Yuzvinskii01}*{Corollary 6.34}.
 
 It was shown by Terao \cite{Terao92} for central arrangements $\A$ that $\mathcal{L}(\A)$ is supersolvable if and only if $\A$ is fiber type, a property generalizing the fiber bundle factorization of braid arrangements. This can be used to show that all such arrangements have rational $K(\pi,1)$ complements. Papadima and Yuzvinsky \cite{PY99} showed that the latter property is \emph{equivalent} to the Koszulness of $\OS(\A)$. 
 
\section{A first counterexample}\label{sec:first_counterexample}
In this section, we prove our main result (\Cref{thm:main}) and use it to construct our first example of a non-supersolvable OS-Koszul arrangement.
This section is tailored to reach this first example as rapidly as possible. Later we will come back to the main theorem and explore its further consequences.
The following purely algebraic statement may be of independent interest.
\begin{Thm}\label{lm:Koszul_lemma}
        Let $A, B$ be  standard graded $\KK$-algebras, with $B$ a graded subalgebra of $A$. Suppose we can find homogeneous elements $a_1, \dots, a_m \in A$ with 
        $a_i B = B a_i$ for each $i$ which induce a direct sum decomposition 
        $A = \bigoplus_i a_i B$.
        If $B$ is Koszul, then $A$ is Koszul if and only if the quotient 
        $C = A/(B_+)$ is Koszul.
        Here $(B_+) = AB_+ = B_+A = AB_+A$.
\end{Thm}
\begin{proof}
The fact that the left, right, and two-sided ideals of $A$ generated by $B_+$ coincide is an immediate consequence of the existence of the basis described in the theorem statement.
To prove the claim, we will show that $C$ has a linear free resolution over $A$; 
this will suffice by \Cref{thm:regularity_Koszul}. 
Let $U$ be the $\KK$-span of $a_1, \dots, a_m$.
Then the multiplication map 
$U \otimes_{\KK} B \to A$ (given on pure tensors by
$w \otimes b \mapsto wb$) is an isomorphism of right $B$-modules.
Since $B$ is Koszul, we can find a linear free resolution (of left $B$-modules)
        \[
            \begin{tikzcd}
            0 & B/B_+\ar[l]
            & B \ar[l]  
            & B(-1)^{\beta_1} 
            \ar[l, "\partial_1" above]
            & B(-2)^{\beta_2}
            \ar[l, "\partial_2" above]
            & \cdots. \ar[l, "\partial_3" above]
            \end{tikzcd}
        \]
        Thinking of elements of $B(-j)^{\beta_j}$ as \emph{row vectors}, each map $\partial_i$ is  right multiplication by a $\beta_{i} \times \beta_{i-1}$ matrix with entries in $B$.
        Applying the (exact) functor $U \otimes_{\KK} -$, we get the long exact sequence
        \[
            \begin{tikzcd}
            0 & U \otimes_{\KK} B/B_+
            \ar[l]
            & A \ar[l]  
            & A(-1)^{\beta_1} 
            \ar[l, "\partial_1" above]
            & A(-2)^{\beta_2}
            \ar[l, "\partial_2" above]
            & \cdots, \ar[l, "\partial_3" above]
            \end{tikzcd}
        \]
        where we now view the $\partial_i$ as matrices with entries in $A$.  In particular, they are $A$-linear maps between graded free left $A$-modules. The last map $A \to U \otimes_{\mathbb{K}} B/B_+$ coincides with the quotient map $A \to A/(B_+) = C$. This therefore gives a linear free resolution of $C$ as a left $A$-module.  
\end{proof}
\noindent
We now show how to apply \Cref{lm:Koszul_lemma} to the problem at hand.
As described in \cite{Yuzvinskii01}*{\S~2.3}, the Orlik--Solomon algebra $\OS(M)$ of a matroid $M$ is naturally graded by $\mathcal{L}(M)$. The element $e_I$ for an independent set $I \subseteq E(M)$ lives in the graded piece of $\OS(M)$ indexed by the flat $\cl_M I$; 
note that $e_D = 0$ in $\OS(M)$ for any dependent set $D \subseteq E(M)$. We write $\OS(M)_F$ for the graded piece indexed by a flat $F$ of $M$.
Multiplication in $\OS(M)$ satisfies $\OS(M)_F \cdot \OS(M)_G \subseteq \OS(M)_{F \vee G}$ for any flats $F,G \in \mathcal{L}(M)$.
For any $F \in \mathcal{L}(M)$, the direct sum $\bigoplus_{G \subseteq F} \OS(M)_G$ is a subalgebra of $\OS(M)$ isomorphic to $\OS(M|_F)$ via the natural inclusion map \cite{Yuzvinskii01}*{Proposition 2.5}. We identify it with $\OS(M|_F)$. 

A similar story can be told for the algebra 
$\OS(\mathbf{d}_x M)$, given an element $x \in E(M)$.
The algebra is now graded by $\mathcal{L}(\mathbf{d}_x M)$; see \Cref{subsec:OS_algebras} for definitions.
Here too, for any $F \in \mathcal{L}(\mathbf{d}_x M)$, the direct sum $\bigoplus_{G \subseteq F} \OS(\mathbf{d}_x M)_G$ is a subalgebra of $\OS(\mathbf{d}_x M)$ isomorphic to $\OS(M|_F)$ and naturally identified with it.

\begin{Prop}\label{lm:Koszul_inherit}
    Let $M$ be a matroid and $F$ a flat in $M$. If $M$ is OS-Koszul then so is $M|_F$.
\end{Prop}
\begin{proof}
    Because of the $\mathcal{L}(M)$-grading, the map $\OS(M) \to \OS(M|_F)$ sending $e_x \mapsto 0$ for all $x \notin F$ is a graded algebra retraction of the inclusion map $\OS(M|_F) \hookrightarrow \OS(M)$.  The conclusion now follows from  \Cref{thm:Koszul_retracts}.
\end{proof}
\noindent
The following result of Terao was mentioned in the introduction.
\begin{Thm}[{\cite{Terao86}*{Theorem 3.8}}]
\label{thm:terao}
    Let $M$ be a matroid and $F$ a modular flat of $M$. Then the multiplication map 
    \[
        \OS(M|_F) \otimes_{\mathbb{Q}} \bigoplus_{G \cap F = \varnothing} \OS(M)_G \to \OS(M).
    \]
    (in which $G$ ranges over the flats of $M$ disjoint from $F$) is an isomorphism of $\mathbb{Q}$-vector spaces.
\end{Thm}
\noindent
To apply \Cref{lm:Koszul_lemma}, we will identify $\OS(M)/(\OS(M|_F)_+)$ with $\rOS(\overline{T}_F(M))$.  
\Cref{thm:terao} shows that $\OS(M)$ is free as a graded $\OS(M|_F)$-module.
Thus the following result lifts the equation of characteristic polynomials in \cite{constructions}*{Corollary 7.4} to a statement about $\QQ$-algebras.
\begin{Lm}\label{lm:quotient_compute}
    Let $M$ be a matroid and $F$ a modular flat of $M$ of positive rank. Then the
    $\mathbb{Q}$-algebras
    $\OS(M)/(\OS(M|_F)_+)$ and $\rOS(\overline{T}_F(M))$ are isomorphic.
\end{Lm}
    \begin{proof}
        Write $E(M) = \{x_1, \dots, x_n\}$ such that $F = \{x_{1}, \dots, x_m\}$ for some $m \leq n$.
        The discussion in \Cref{subsec:OS_algebras} gives the isomorphism $\rOS(\overline{T}_F(M))
        \cong \OS(\mathbf{d}_{x_1}\overline{T}_F(M))$.
        We will prove the isomorphism $\OS(M)/(\OS(M|_F)_+) \cong \OS(\mathbf{d}_{x_1}\overline{T}_F(M))$.
        Both $\OS(M)$ and $\OS(\mathbf{d}_{x_1}\overline{T}_F(M))$ are canonically quotients of the exterior algebra 
        $\Lambda(E) = \ext{e_{1}, \dots, e_n}$ in which the generator $e_i$ corresponds to the element $x_i \in E$. By comparing the kernel ideals (as in \Cref{subsec:OS_algebras}), one sees readily that $\OS(\mathbf{d}_{x_1}\overline{T}_F(M))$
        is a quotient of $\OS(\overline{T}_F(M))$, which is a quotient of $\OS(M)$. For $1 \leq i \leq m$, the pair $\{x_1, x_i\}$ is a circuit of $\overline{T}_F(M)$, so $e_1 - e_i = 0$ in $\OS(\mathbf{d}_{x_1}\overline{T}_F(M))$.
        We also have $e_1 = 0$ in this ring. It follows that $\OS(\mathbf{d}_{x_1}\overline{T}_F(M))$ is a quotient of $\OS(M)/(e_1, \dots, e_m) = \OS(M)/(\OS(M|_F)_+)$. We compare their dimensions as $\mathbb{Q}$-vector spaces. The $\mathcal{L}(\mathbf{d}_{x_1}\overline{T}_F(M))$-grading gives
        \[
            \dim_{\mathbb{Q}} \OS(\mathbf{d}_{x_1}\overline{T}_F(M)) = \sum_{G \in \mathcal{L}(\mathbf{d}_{x_1}\overline{T}_F(M))} \dim_{\mathbb{Q}} \OS(\mathbf{d}_{x_1}\overline{T}_F(M))_G
            = \sum_{\substack{G \in \mathcal{L}(M)\\G \cap F = \varnothing}}
            \dim_{\KK} \OS(M)_G,
        \]
        using the isomorphisms 
        $\OS(\mathbf{d}_{x_1}\overline{T}_F(M))_G \cong 
        \OS(\overline{T}_F(M)|_G)_G 
        = \OS(M|_G)_G \cong \OS(M)_G$. Comparing the last sum with \Cref{thm:terao} shows that 
        $\dim_{\mathbb{Q}} \OS(M)/(\OS(M|_F)_+) = \dim_{\mathbb{Q}} \OS(\mathbf{d}_{x_1}\overline{T}_F(M))$.
        Since one is a quotient of the other, they must be isomorphic.
    \end{proof}

We can now prove that $M$ is OS-Koszul if and only if the same is true of $M|_F$ and $\overline{T}_F(M)$.
\begin{proof}[{Proof of \Cref{thm:main}}]
      By \Cref{lm:Koszul_inherit}, we may assume that $M|_F$ is OS-Koszul. The result then follows from  \Cref{lm:Koszul_lemma} by substituting $\OS(M)$ for $A$ and $\OS(M|_F)$ for $B$.
      The existence of the prescribed homogeneous basis is ensured by \Cref{thm:terao}: we take the union of $\mathbb{Q}$-bases in each graded piece $\OS(M)_G$.
      That $a_i B = B a_i$ for all $i$ is automatic in graded-commutative rings. By \Cref{lm:quotient_compute}, $A/(B_+) \cong \rOS(\overline{T}_F(M))$, which is Koszul if and only if $\overline{T}_F(M))$ is OS-Koszul.  This completes the proof.
    \end{proof}
\noindent
We put \Cref{thm:main} to work by giving a first example of a non-supersolvable OS-Koszul matroid.
\begin{figure}[htb!]
        \centering
        \tikzset{
  LabelStyle/.style = { rectangle, rounded corners, draw,
                        minimum width = 2em},
  VertexStyle/.append style = { inner sep=5pt,
                                font = \Large\bfseries},
  EdgeStyle/.append style = {bend left} }
        \begin{tikzpicture}
        \tikzset{vertex/.style={circle,fill=blue!25,minimum size=12pt,inner sep=2pt}}
  \tikzset{every loop/.style={}}
    \node[vertex] (A) at (-4,0)  [shape=circle,draw=black,fill=gray] {};
        \node[vertex] (B) at (0,2)  [shape=circle,draw=black,fill=gray] {};
    \node[vertex] (C) at (0,-2)  [shape=circle,draw=black,fill=gray] {};
    \node[vertex] (D) at (4,0)  [shape=circle,draw=black,fill=gray] {};
    \path[loop/.style={looseness=30}]   (A) edge [loop left, thick] node {$a$} (A);
    \path[loop/.style={looseness=30}]   (B) edge [loop above, thick] node {$d$} (B);
    \path[loop/.style={looseness=30}]   (C) edge [loop below, ultra thick] node {$k$} (C);
    \path[loop/.style={looseness=30}]   (D) edge [loop right, ultra thick] node {$n$} (D);
    \path (B) edge [thick] node[below]{$-$} node[above]{$c$} (A);
    \path (A) edge [bend left, thick] node[above]{$+$} node[below]{$b$} (B);
     \path (C) edge [thick] node[above]{$-$} node[below]{$g$} (A);
    \path (A) edge [bend right, thick] node[below]{$+$} node[above]{$f$} (C);
    \path (B) edge [bend left, thick] node[right]{$-$} node[left]{$i$} (C);
    \path (C) edge [bend left, thick] node[left]{$+$} node[right]{$h$} (B);
   \path (C) edge [ultra thick] node[above]{$-$} node[below]{$l$} (D); 
   \path (D) edge [bend left, ultra thick] node[below]{$+$} node[above]{$m$} (C);
   \path (D) edge [thick] node[above]{$j$} node[below]{$+$} (B);
  \draw [rounded corners, thick] (D)  arc(0:180:4) (A);
  \node[below] at (90:4) {$e$};
  \node[above] at (90:4) {$+$};
    \node[vertex] (A) at (-4,0)  [shape=circle,draw=black,fill=gray] {};
    \node[vertex] (D) at (4,0)  [shape=circle,draw=black,fill=gray] {};
\end{tikzpicture}
        \caption{A supersolvable signed-graphic matroid of rank 4 containing a modular flat (in bold) which does not lie any full modular flag.}
        \label{fig:signed_graphic_matroid}
\end{figure}
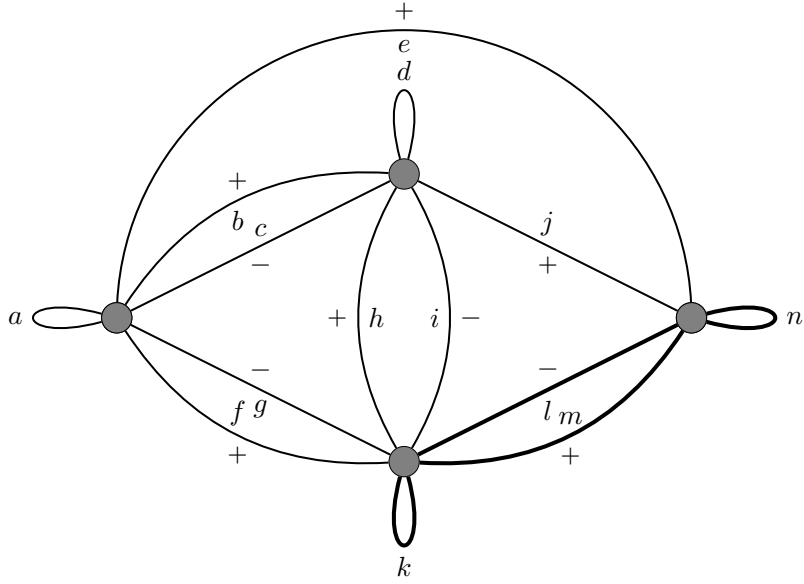
    
\begin{Ex}\label{ex:rank_3_Koszul_nonSS}
    As a consequence of \Cref{thm:main}, if a matroid $M$ is supersolvable then the complete principal truncation of $M$ along any modular flat is OS-Koszul. Let $M$ be the signed-graphic matroid of rank 4 whose signed graph is depicted in \Cref{fig:signed_graphic_matroid}.
    It is readily verified (for instance, using the \emph{Macaulay2} program \cite{M2} and criterion \ref{it:complement} in \Cref{df:modular}) that
    \[
        \varnothing \subsetneq \{a\} \subsetneq \{a,b,c,d\} \subsetneq 
        \{a,b,c,d,f,g,h,i,k\} \subsetneq 
        \{a,b,c,d,e,f,g,h,i,j,k,l,m,n\}
    \]
    represents a full modular flag in $M$.
    By a similar verification, the rank 2 flat $F = \{k,l,m,n\}$ is modular.
    The simplified matroid $\si(\overline{T}_F(M))$ is depicted in \Cref{fig:rank_3_non_ss} as a point-line configuration (see \cite{Oxley}*{\textsection~1.5}).
    The figure reveals that $\si(\overline{T}_F(M))$ has no modular coatoms (i.e.\ no line intersects all others). 
    It is realized over $\QQ$ by the arrangement with defining linear forms:\\
    \noindent 
    \begingroup
    \setlength{\tabcolsep}{4pt} 
    \begin{gather*}
    \ell_a = x, \hspace{1.5em} 
    \ell_b = x-y, \hspace{1.5em}
    \ell_c = x+y, \hspace{1.5em}
    \ell_d = y,  \hspace{1.5em}
    \ell_e = x-2z, \hspace{1.5em}
    \ell_f = x-z, \hspace{1.5em}
    \ell_g = x+z, \\
    \ell_h = y-z, \hspace{1.5em}
    \ell_i = y+z, \hspace{1.5em}
    \ell_j = y-2z, \hspace{1.5em}
    \ell_k = z.\\
    \end{gather*}
    \endgroup
    
    \begin{figure}[hbt!]
        \centering
        \newcommand{\edgelabeledgraph}[3][scale = 1.2]{
	\begin{tikzpicture}[#1]          
	\newcommand*\points{#2}     
	\newcommand*\edges{#3}          
	\newcommand*\scale{0.75}          
	\foreach \x/\y/\z/\w/\a in \points {
		\draw[fill = black!50] (\scale*\x,\scale*\y) circle [radius = 0.1] node[label = {[label distance = 0.05 cm]\a: $\w$}] (\z) {}; 
	}
	\foreach \x/\y/\p/\a/\l in \edges { \draw (\x) -- (\y) node [pos = \p, sloped, \a] {{\small $\l$}}; }      
	\end{tikzpicture}
}
\edgelabeledgraph{
    -4/0/b/b/270,
    0/0/e/e/270,
    4/0/j/j/270,
    .4/1.2/g/g/300,
    1/3/f/f/120,
    2/6/k/k/90,
    2.67/4/h/h/30,
    3.33/2/i/i/30,
    .77/2.31/a/a/180,
    2.89/3.33/d/d/30,
    1.64/2.73/c/c/270
    }{b/e/.5/below/,
    e/j/.5/below/,
    j/i/.5/below/,
    i/d/.5/below/,
    d/h/.5/below/,
    h/k/.5/below/,
    e/g/.5/below/,
    g/a/.5/below/,
    a/f/.5/below/,
    f/k/.5/below/,
    b/a/.5/below/,
    a/c/.5/below/,
    d/c/.5/below/,
    b/f/.5/below/,
    f/h/.5/below/,
    f/c/.5/below/,
    c/i/.5/below/,
    g/c/.5/below/,
    c/h/.5/below/,
    b/g/.5/below/,
    g/i/.5/below/
}
        \caption{A rank 3 OS-Koszul matroid that is not supersolvable}
        \label{fig:rank_3_non_ss}
    \end{figure}
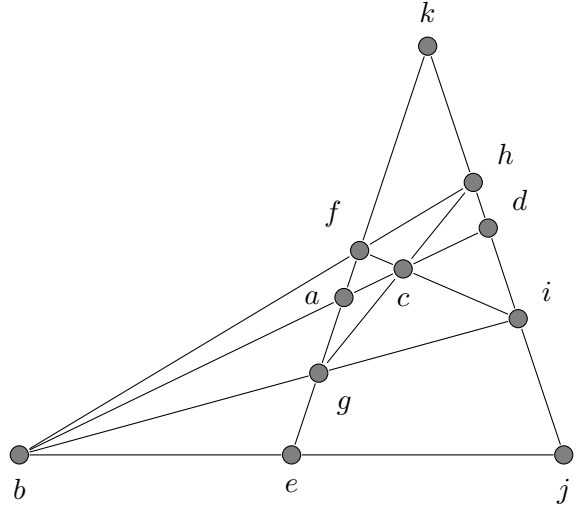
    \noindent
    Thus we have an irreducible, OS-Koszul, rank 3 arrangement defined over $\QQ$ with non-supersolvable intersection lattice.
    We remark that it is also free by a simple Macaulay2 \cite{M2} computation; see \Cref{sec:questions} for a definition.
    In the sections to follow, we explain how this example was obtained and how to systematically use \Cref{thm:main} to construct non-supersolvable OS-Koszul matroids.
\end{Ex}

\section{A construction using Dowling Geometries}\label{sec:modular_no_flag}
In this section, we describe a general procedure to produce a supersolvable matroid possessing a modular flat through which passes no full modular flag.
As we shall see, such matroids are vital to all of our constructions.
A binary example is given in \cite{Z91}*{Example 4.3}, but this example is not realizable over $\CC$.
Here, we construct $\CC$-realizable examples using Dowling geometries as a starting point. A general definition of Dowling geometries is given in \cite{Oxley}*{\S~6.10}. We will not need the general framework and instead define explicitly the case we need: Dowling geometries over cyclic groups. These are precisely the $\CC$-realizable ones.
The fact that the following definition is a special case of Dowling's follows from \cite{dowling}*{Corollary 10.1}. 
\begin{Df}\label{df:dowling}
    Let $m, r$ be positive integers. Let $\eta$ be a primitive $m$-th root of unity. The \textbf{cyclic Dowling arrangement} of order $m$ and rank $r$, denoted $Q_r(m)$, is the arrangement in $\mathbb{C}^r$
    \[
        Q_r(m) = \{ x_i \mid 1 \leq i \leq r\} \cup \{x_i - \eta^k x_j \mid 1 \leq i < j \leq r, \; 0 \leq k < m\}.
    \]
    The \textbf{cyclic Dowling geometry} of order $m$ and rank $r$ is the matroid $MQ_r(m) \deq M(Q_r(m))$.
    For a subset $Z$ of $\{1,\dots,m\}$, the \textbf{Dowling subarrangement} induced by $Z$ is the subarrangement 
    \[
        Q_Z(m) \deq \{x_i \mid i \in Z\} \cup \{x_i - \eta^k x_j \mid i, j \in Z,\; i < j,\; 0 \leq k < m\}.
    \]
    The corresponding submatroid $MQ_Z(m) \deq M(Q_Z(m))$ of $MQ_r(m)$ is the \textbf{Dowling subgeometry} induced by $Z$. It is isomorphic to 
    $Q_{\abs{Z}}(m)$ and its ground set is a flat of $Q_r(m)$. We also call ``Dowling subgeometry'' the flat $E(M(Q_Z(m )))$ of any submatroid of $MQ_r(m)$ containing it.
\end{Df}
\noindent
Note that for $m \leq 2$, this arrangement is realized over $\mathbb{Q}$. 
For $m = 0$, it is the Boolean arrangement. For $m = 1$, it is the braid arrangement.
For $m = 2$, it is the complete signed-graphic arrangement.
The following proposition provides us with an ample number of modular flats in $MQ_r(m)$.
\begin{Prop}[{\cite{dowling}*{Theorem 4}}]\label{prop:modular_in_Dow}
The ground set of any Dowling subgeometry of $MQ_r(m)$ is a modular flat in $MQ_r(m)$.
\end{Prop}
\noindent
By \Cref{prop:modular_in_Dow}, the nested ground sets of the Dowling subgeometries 
$MQ_{\{1,\dots,r'\}}(m)$ as $r'$ varies between $0$ and $r$ form a full modular flag in $MQ_r(m)$.
Consequently:
\begin{Cor}[Dowling]\label{cor:Dow_ss}
    For any $m, r \geq 1$, the matroid $MQ_r(m)$ is supersolvable.
\end{Cor}
\noindent
Dowling geometries themselves do not have the
property we desire. However, we only need a small modification to get there.
We will construct for ranks $r, k$ a supersolvable rank $r$ matroid with a rank-$k$ modular flat not contained in any full modular flag. We cannot have $k = 0$, $k=r-1$, or $k = r$. All intermediate values are possible.
\begin{Lm}\label{lm:Dow_construction}
    If $m \geq 2$ and 
    $1 \leq k \leq r - 2$, then there exists a round, supersolvable spanning submatroid $M$ of 
    $MQ_r(m)$ containing as a modular flat a rank-$k$ Dowling subgeometry which is not contained in any rank-$(k+1)$ modular flat of $M$.
\end{Lm}
\begin{proof}
    We identify the elements of $MQ_r(m)$ with the corresponding linear forms as given in \Cref{df:dowling}.
    Let $M$ be the submatroid of $MQ_r(m)$ obtained by deleting the elements $x_1 - x_i$ for $k < i \leq r$. 
    Let $F$ be the ground set of the Dowling subgeometry $MQ_{\{1,\dots,k\}}(m)$. By \Cref{prop:modular_in_Dow}, $F$ is a modular flat of rank $k$ in $MQ_r(m)$ and therefore also in $M$, by \Cref{prop:basic}\eqref{it:deletion}.
    Similarly, the ground set $H$ of 
    $MQ_{\{2,\dots,r\}}(m)$ is a modular flat in $M$. The restriction $M|_H$ is supersolvable by \Cref{cor:Dow_ss} since it is isomorphic to 
    $MQ_{r-1}(m)$. Since $H$ is a coatom in $M$, it follows by \Cref{lm:coatom_supersolvable} that $M$ is supersolvable. To see that $M$ is round of rank $r$, we observe that 
    $x_1, \dots, x_r$ is a basis of $M$ and that for all $1 \leq i < j \leq n$, $M$ contains the 3-element circuit 
    $\{x_i, x_j, x_i - \eta x_j\}$. This suffices by \Cref{prop:round_criterion}.

    It remains to show that $F$ is not contained in any modular flat of $M$ of rank $k+1$. Suppose $G$ were some such modular flat. Let $g \in G - F$ be any element. If $g = x_\ell - \eta^j x_m$ for $\ell \leq k < m$ and some $j$, then the element 
    $g' = x_m$ also lies in $G$. 
    Thus we may assume that either 
    $g = x_{\ell}$ for some $\ell > k$ or else
    $g = x_{\ell} - \eta^j x_{m}$ for some 
    $\ell, m > k$. We deal with the second case first.
    In this case,
    $F \cup \{g\}$ is a flat of $MQ_r(m)$, hence of $M$, whence $G = F \cup \{g\}$ and $g$ is a coloop in $M|_F$. This means that $M|_F$ is disconnected. Since $M$ is connected and $F$ is modular in $M$, this contradicts \Cref{prop:basic}\eqref{it:connected}.
    Now we return to the first case. 
    That is, $g = x_{\ell}$. By symmetry, we may assume $\ell = {k+1}$. Then $G = E(MQ_{\{1,\dots,k+1\}}) \cap E(M)$.
    The intersection 
    $K = E(M) \cap E(MQ_{\{1,k+1,k+2\}}(m))$ is a flat of $M$ of rank 3. 
    The intersection $K \cap G$ is a modular flat in $M|_K$ by \Cref{prop:basic}\eqref{it:restriction}.
    It is
    \[
        K \cap G =
        E(MQ_{\{1,k+1\}}) \cap E(M) =
        \{x_1, x_{k+1}\} \cup \{x_1 - \eta^j x_{k+1} \mid 1 \leq j < m\},
    \]
    which has rank 2.
    We derive a contradiction by showing that this flat is not modular in $M|_K$. For this, consider the line in $M|_K$ spanned by $x_1 - \eta x_{k+1}$ and 
    $x_{k+1} - \eta x_{k+2}$. In $MQ_{r}(m)$, this line intersects $K \cap G$ in the unique point $x_1 -  x_{k+1}$, which is missing from $M$. Since $M|_K$ has rank 3, this shows that $K \cap G$ is not modular in $M$.
\end{proof}
\noindent
We now investigate modularity and supersolvability of complete principal truncations.
\begin{Lm}\label{lm:modular_in_truncation}
    Let $M$ be a round matroid and $F \in \mathcal{L}(M)$ a modular flat of rank $\geq 2$.
    The modular flats of rank $\geq 2$ in $\overline{T}_F(M)$ are precisely the modular 
    flats of $M$ containing $F$.
\end{Lm}
\begin{proof}
    Let $G \in \mathcal{L}(\overline{T}_F(M))$ be a modular flat of rank at least 2.
    Suppose first that $G$ contains $F$. We show that $G$ is modular in $\overline{T}_F(M)$ if and only if it is modular in $M$. Let $H$ be any flat of $\overline{T}_F(M)$. We consider the equality
        \[\tag{$\star$}\label{eq:mod}
            \rk_{\overline{T}_F(M)} H + \rk_{\overline{T}_F(M)} G
            = 
            \rk_{\overline{T}_F(M)} (H \cap G) + \rk_{\overline{T}_F(M)} (H \vee G).
        \]
        If $H$ contains $F$, then by \Cref{lm:rank_function} this reads
        \[
            (\rk_{M} H - \rk_M F + 1) + (\rk_{M} G- \rk_M F + 1)
            = 
            (\rk_{M} (H \cap G)- \rk_M F + 1) + (\rk_{M} (H \vee G)- \rk_M F + 1).
        \]
        If on the other hand, $H \cap F = \varnothing$, then \eqref{eq:mod} reads
        \[
            \rk_M H + (\rk_{M} G- \rk_M F + 1)
            = 
            \rk_M (H \cap G) + (\rk_{M} (H \vee G)- \rk_M F + 1).
        \]
        In either case, \eqref{eq:mod} is equivalent to
        \[
            \rk_M H + \rk_M G = \rk_M (H \cap G) + \rk_M(H \vee G).
        \]
        If $G$ is modular in $M$, then the latter equality holds for all $H$, showing that $G$ is modular in $\overline{T}_F(M)$. 
        If $G$ is \emph{not} modular in $M$, then by \Cref{df:modular}\ref{it:complement}, equality fails for some flat $H$ of $M$ disjoint from $G$. Then $H$ is disjoint from $F$, meaning it is a flat of $\overline{T}_F(M)$. Since \eqref{eq:mod} fails to hold for this choice of $H$, we conclude that $G$ is \emph{not} modular in $\overline{T}_F(M)$.

        Next suppose that $G$ is disjoint from $F$. 
        Let $H$ be a flat of $M$ maximal among flats disjoint from both $F$ and $G$.
        Let $F' = F \vee H$ and $G' = G \vee H$.
        Then $F'- H$ and $G' - H$ are flats of the matroid $M/H$.
        If some element $x \in E(M/H)$ lay outside both of these flats, then
        $H' = \cl_M(H \cup \{x\})$ would be a strictly bigger flat of $M$ still disjoint from both $F$ and $G$. It must therefore be the case that $E(M/H) = (F' - H) \cup (G' - H)$, i.e.\ $E(M) = F' \cup G'$. Roundness of $M$ forces one of $F' = E(M)$ or $G' = E(M)$ to hold.
        We must then have $\rk_M H \geq \rk M - \rk_M F$ in the first case and $\rk_M H \geq \rk M - \rk_M G$ in the second.
        Since $H$ is disjoint from $F$, it is a flat of $\overline{T}_F(M)$. We
        estimate the sum $\rk_{\overline{T}_F(M)} H + \rk_{\overline{T}_F(M)} G = \rk_M H + \rk_M G$.
        In the first case, we have
        \[
            \rk_M H + \rk_M G
            \geq \rk M - \rk_M F + 2
            = \rk\overline{T}_F(M) + 1.
        \]
        In the second case, we have
        \[
            \rk_M G + \rk_M H
            \geq
            \rk M = \rk \overline{T}_F(M) + \rk_M F - 1
            \geq \rk \overline{T}_F(M) + 1.
        \]
        In either case, we get $\rk_{\overline{T}_F(M)} H + \rk_{\overline{T}_F(M)} G > \rk \overline{T}_F(M)$.
        Since $F$ and $G$ are disjoint, this contradicts the supposed modularity of $G$ in $\overline{T}_F(M)$, by \Cref{df:modular}\ref{it:complement}.
\end{proof}
    
\begin{Cor}\label{cor:truncation_not_ss}
    Let $M$ be a round matroid of rank $r$ and $F$ a flat of rank $s \geq 2$. Then $\overline{T}_F(M)$ is supersolvable if and only if $M$ contains modular flats 
    $F = F_s \subsetneq F_{s+1} \subsetneq \dots \subsetneq F_r$ with $\rk_M F_i = i$.
    In particular, if $M|_F$ is supersolvable, then $\overline{T}_F(M)$ is supersolvable if and only if $M$ has a full modular flag passing through $F$.
\end{Cor}
\noindent
The final ingredient we need is the following, which we use to ensure realizability.
\begin{Prop}
\label{prop:truncation_realizable}
    Let $M$ be a matroid and $F$ a modular flat of $M$. If $M$ is $\CC$-realizable (resp.\ $\RR$-realizable, $\QQ$-realizable) then so is $\overline{T}_F(M)$.
\end{Prop}
\begin{proof}
    Let $\A$ be a central hyperplane arrangement in $\CC^r$ with $M(\A) \cong M$. We identify hyperplanes in $\A$ with the corresponding elements of $M$. 
    Let $X = \bigcap_{H \in F} H \in \mathcal{L}(\A)$ be the ``geometric'' flat corresponding to $F$, and let 
    $s = \codim X = \rk_M F$.
    Let $Y$ be a generic subspace of $\mathbb{C}^r$ of codimension $s-1$ \emph{containing} $X$.  We identify $Y$ with $\CC^{r-s+1}$ via an affine isomorphism
    and consider on it the arrangement $\mathcal{B} = \{Y \cap H \mid H \in \A\}$.
    It follows from the usual description of $\overline{T}_F(M)$ as an iterated principal truncation (see e.g.\ \cite{Brylawski86}*{pp.\ 148--9}) that $\mathcal{B}$ realizes $\overline{T}_F(M)$; c.f.\ \ \cite{falk_proudfoot}*{Theorems 2.1 and 2.4} and \cite{Denham14}*{Proposition 2.3}.

    If $M$ is $\RR$-realizable (resp.\ $\QQ$-realizable), we choose the hyperplanes in $\A$ to be defined over $\RR$ (resp.\ $\QQ$) and $Y$ generic containing $X$ and defined over $\RR$ (resp.\ $\QQ$).
\end{proof}
\noindent
We are now in a position to explain how we found \Cref{ex:rank_3_Koszul_nonSS}: we constructed the arrangement in \Cref{fig:signed_graphic_matroid} in accordance with \Cref{lm:Dow_construction}. The complete principal truncation of the flat in bold was guaranteed to be non-supersolvable by 
\Cref{cor:truncation_not_ss} and $\QQ$-realizable by \Cref{prop:truncation_realizable}.
We now generalize this argument to prove \Cref{thm:full_counterexample}.

\begin{proof}[{Proof of \Cref{thm:full_counterexample}}]
Let $M$ be a matroid satisfying the conditions in \Cref{lm:Dow_construction} with parameter values $m = 2$, $k = 2$, $r \geq 4$, and let $F$ be the rank 2 flat ensured by the lemma. 
Let $N = \overline{T}_F(M)$.
By \Cref{lm:modular_in_truncation}, the only modular flats of rank $2$ in $N$ are the ones which contain $F$ and are modular (of rank 3) in $M$. Since there are no such flats, $N$ has no modular flats of rank $2$.
Thus $N$ is not supersolvable. 
By \Cref{thm:main}, $N$ is OS-Koszul. 
Since $m = 2$, the Dowling geometry 
$Q_r(m)$ is realizable over $\mathbb{Q}$, hence so is the submatroid $M$, and so also is the complete principal truncation $N$ by \Cref{prop:truncation_realizable}.
Thus the arrangement realizing $N$ fulfills the requirements in the statement of \Cref{thm:full_counterexample}.
Since $r \geq 4$ was chosen arbitrarily and $N$ has rank $r-1$, this proves the claim.
The matroids thus constructed are round by \Cref{prop:truncation_round}.
\end{proof}

\section{Supersolvability of parallel connections}\label{sec:modular_joins}

In this section we discuss the question of when the parallel connection of two matroids is supersolvable. This question has already been answered in the literature, in a paper by Ziegler \cite{Z91}. Unfortunately, the answer provided there is incorrect. The ideas of Ziegler's proof can nevertheless be rescued to arrive at the correct answer. This is our task in this section.
See \Cref{rmk:error} for a discussion of the error in Ziegler's statement.

Following Ziegler, we will prove a result in the more general setting of {modular joins}, which were introduced in \cite{Z91}.
These are the most well-behaved among {generalized parallel connections}, which were first defined in \cite{constructions}.

\begin{Df}\label{df:GPC}
    Let $E$ be a finite set.
    Let $E = E_0 \sqcup E_1 \sqcup E_2$ be a tripartition of $E$.
    Write $E_{01} = E_0 \cup E_1$ and $E_{12} = E_1 \cup E_2$.
    Suppose $M$ is a matroid on $E_{01}$ and $N$ is a matroid on $E_{12}$ such that
    $E_1$ is a modular flat of $M$ and
    $M|_{E_1} = N|_{E_1}$.
    
    One defines the \textbf{generalized parallel connection} of $M$ and $N$ along $E_1$---denoted $P_{E_1}(M, N)$---to be the unique matroid on $E$ whose flats are the subsets $F \subseteq E$ such that 
    $F \cap E_{01}$ is a flat of $M$ and 
    $F \cap E_{12}$ is a flat of $N$.
    The rank of this matroid is $\rk P_{E_1}(M, N) = \rk M + \rk N - \rk_M E_1$.
    When $E_1$ is modular in $N$ also, we say (following Ziegler) that $M$ is the \textbf{modular join} of $M$ and $N$ along $E_1$.
    In this case, $P_{E_1}(N, M)$ also exists and is equal to $P_{E_1}(M, N)$.
\end{Df}
\noindent
The fact that $P_{E_1}(M, N)$ defines a matroid follows from \cite{constructions}*{Theorem 5.3} and the discussion preceding that theorem
(see also \cite{Z91}*{Definition 3.1}).
Some special cases of the construction are worth highlighting.
When $E_0 = \varnothing$, the connection is equal to $N$. Similarly, it is equal to $M$ if $E_2 = \varnothing$. These two cases are \emph{trivial} joins. The connection is \emph{non-trivial} if $E_0, E_2$ are both non-empty.
When $E_1 = \varnothing$, this connection is simply the direct sum of $M$ and $N$.
When $M$ and $N$ are simple and $E_1 = \{e\}$ is a singleton,\footnote{There are simplifying assumptions; parallel connections can be defined without these constraints.} we write $P_{e}(M, N)$ for the connection (which is a modular join!) and call it the \emph{parallel connection} of $M$ and $N$ along $e$. 

The construction can be performed on matroids on unrelated ground sets as well. Suppose $M$ and $N$ are matroids on disjoint ground sets, $F$ is a modular flat of $M$, and $Y \subseteq E(N)$ is a subset such that $M|_F$ and $N|_{Y}$ are isomorphic. Choosing a bijection $\psi: F \to Y$ inducing such an isomorphism, we can identify $F$ with $Y$ and construct a generalized parallel connection along the tripartition $E(M) - F, F, E(N) - Y$ of the fused set $E(M) \cup_{\psi} E(N)$. We write $P_{F, Y}(M, N)$ for the resulting matroid. The dependence on $\psi$ is left implicit.
When $M$ and $N$ are simple and $F = \{f\}$, $Y = \{y\}$ are singletons, we use the notation $P_{f, y}(M, N)$.
Of course, the bijection $\psi$ is uniquely determined in this case.

\begin{Prop}[{\cite{constructions}*{Proposition 5.10(5)}}]\label{prop:gpc_connected}
    Let $M = P_F(M', M'')$ be a generalized parallel connection. If $M'$ is connected, then $M$ is connected if and only if $M''$ is connected.
\end{Prop}
\begin{Lm}\label{lm:modular_in_gpc}
    Let $M = P_{F}(M', M'')$ be a generalized parallel connection. If $G$ is a modular flat of $M'$  containing $F$ then $G \cup E(M'')$ is a modular flat of $M$.
\end{Lm}
\begin{proof}
    By \Cref{df:modular}\ref{it:complement}, it suffices to show that 
    $\rk_{M}(G \cup E(M'')) + \rk_M H \leq \rk M$ for any flat $H$ of $M$ disjoint from $G \cup E(M'')$.
    We must then have $H \subseteq E(M')$. We compute
    \begin{align*}
        \rk_M(G \cup E(M'')) + \rk_M H &\leq
        \rk_{M'} G + \rk M'' - \rk_M F + \rk_{M'} H\\
        &\leq \rk M' + \rk M'' - \rk_M F = \rk M.
    \end{align*}
    We have used the submodularity of $\rk_{M}$ in the first line and the modularity of $G$ in the guise of \Cref{df:modular}\ref{it:complement} in the second.
\end{proof}
\begin{Cor}[{\cite{constructions}*{Proposition 5.10(1)}}]\label[Cor]{prop:side_modular}
    $E(M'')$ is a modular flat of $P_F(M', M'')$.
\end{Cor}
\noindent
In the remainder of this section, we restrict ourselves to modular joins. 
In this setting, \Cref{prop:side_modular}, combined with \Cref{prop:basic}\eqref{it:intersection}, implies:
\begin{Cor}\label{cor:modular_parts}
    Let $M = P_F(M', M'')$ be a modular join. Then $E(M')$, $E(M'')$, and $F$ are all modular flats of $M$.
\end{Cor}

\begin{Lm}\label{lm:gpc_behaviour}
    With the notation of \Cref{df:GPC}, but assuming a modular join, we have
    \begin{enumerate}[(1)]
        \item $\si(P_{E_1}(M, N) = P_{\overline{E_1}}(\si(M), \si(N))$ where $\overline{E_1}$ is the ground set of $\si(M|_{E_1}) = \si(N|_{E_1})$.
        \label{it:simplify}
        \item If $S \subseteq E$ is any subset containing $E_1$, then $P_{E_1}(M,N)|_S = P_{E_1}(M|_{S \cap E_{01}}, N|_{S \cap E_{12}})$.\label{it:restrictify}
        \item If $F \subseteq E_1$ is a flat of $M|_{E_1} = N|_{E_1}$, then 
        $P_{E_1}(M, N)/F = P_{E_1 - F}(M/F, N/F)$.\label{it:contractify}
    \end{enumerate}
\end{Lm}
\begin{proof}
    Each case follows by a straightforward comparison of the flats on either side of the equality. The necessary modularities follow in (1) from the fact that modularity is a property in the lattice of flats of a matroid. In (2) they follow from \Cref{prop:basic}\eqref{it:deletion}.
    In (3) they follow from \Cref{prop:basic}\eqref{it:quotient}.
\end{proof}
\noindent
Suppose $M$ is a matroid with modular flats $F, F'$ such that $M = P_{F\cap F'}(M|_F, M|_{F'})$. 
In such a case, we will say that $M$ is the modular join of $F$ and $F'$.
We also say that $M$ is a modular join \emph{over} $F \cap F'$.
The following lemmas allow us to recognize such cases.
\begin{Prop}[{\cite{Z91}*{Proposition 3.3}}]\label{prop:ziegler}
    Let $M$ be a matroid on a ground set $E$. Let $F, F'$ be modular flats of $M$ such that $E = F \cup F'$. Then $M$ is the modular join of $F$ and $F'$.
\end{Prop}
\begin{Prop}\label{prop:magic}
    Let $M$ be a matroid. Let $F \subsetneq E(M)$ be a modular flat. Then $M$ is a non-trivial modular join over $F$ if and only if $M/F$ is disconnected.
\end{Prop}
\begin{proof}
    Necessity follows from \Cref{lm:gpc_behaviour}\ref{it:contractify}, which implies that 
    $P_{F}(M_1, M_2)/F = M_1/F \oplus M_2/F$. Sufficiency follows from Lemmas \ref{prop:modular_descent} and \ref{prop:ziegler}.
\end{proof}
\begin{Lm}\label{lm:sufficiency}
    Let $M$ be a matroid which is the modular join of flats $F_1, F_2$. 
    Suppose $M_1 \deq M|_{F_1}$ and $M_2 \deq M|_{F_2}$ are supersolvable and further that some full modular flag of $M_1$ passes through $F \deq F_1 \cap F_2$. Then $M$ is supersolvable.
\end{Lm}
\begin{proof}
    Let $\varnothing = G_0 \subsetneq G_1 \subsetneq \dots \subsetneq G_p = F_2$ be a full modular flag in $M_{2}$.
    Similarly, let 
    $\varnothing = H_0 \subsetneq H_1 \subsetneq \dots \subsetneq H_q = F_1$ be a full modular flag in $M_1$ in which $H_s = F$ for some $0 \leq s \leq q$.
    We note that $G_i$ is a modular flat of $M$ for each $i$ by the combination of \Cref{prop:side_modular} and \Cref{prop:basic}\eqref{it:transitivity}.
    By \Cref{lm:modular_in_gpc}, $H_i \cup F_2$ is modular in $M$ for all $i \geq s$. Hence the spliced sequence of flats
    \[
        \varnothing = G_0 \subsetneq G_1 \subsetneq \dots \subsetneq G_p = F_2 = H_s \cup F_2 \subsetneq H_{s+1} \cup F_2 \subsetneq \dots \subsetneq H_q \cup F_2 = F_1 \cup F_2 = E(M)
    \]
    is a full modular flag in $M$.
\end{proof}
The above proposition has a converse, provided the modular join is \emph{minimal} in a suitable sense.
\begin{Thm}\label{thm:supersolvable_GPC}
Let $M$ be a matroid. Suppose $F$ is a flat of $M$ minimal for the property that $M$ is a non-trivial modular join over $F$. Choose flats $F_0, F_1$ with $F = F_0 \cap F_1$ such that $M = P_F(F_0, F_1)$. 
Then $M$ is supersolvable if and only if $M_0 \deq M|_{F_0}$ and $M_1 \deq M|_{F_1}$ are both supersolvable and at least one of them has a full modular flag passing through $F$.
\end{Thm}
\begin{proof}
    Sufficiency is given by \Cref{lm:sufficiency}. Now suppose $M$ is supersolvable. Then so are $M_0$ and $M_1$ by \Cref{prop:inherit_ss}. It remains to prove the claim about $F$. There is no loss of generality in assuming that $M$ is simple because the relevant construction and properties operate on the level of lattices of flats.

    Let $H$ be a modular coatom of $M$. We start by showing that $H$ contains one of $F_0$ or $F_1$.
    Suppose we could choose elements 
    $x_0 \in F_0 - (H \cup F)$ and 
    $x_1 \in F_1 - (H \cup F)$. Then $\{x_0, x_1\}$ would be a rank 2 flat disjoint from $H$, contradicting \Cref{df:modular}\ref{it:complement}. 
    We deduce that $F_i - F \subseteq H$ for $i = 0$ or $i = 1$. It holds for $i = 0$ without loss of generality. Now $F_0' \deq F_0 \cap H$ is modular in $M$ by \Cref{cor:modular_parts} and \Cref{prop:basic}\eqref{it:intersection}. Hence $M$ is the (non-trivial) modular join of $F_0'$ and $F_1$ by \Cref{prop:ziegler}. 
    Since $F_0' \cap F_1 = H \cap F \subseteq F$,
    the minimality of $F$ forces $F = H \cap F$. So $F_0 \subseteq H$.

    Let $F_1' \deq F_1 \cap H$.
    Suppose $F_1' = F$. By \Cref{prop:basic}\eqref{it:restriction}, $F$ is a modular coatom in the supersolvable matroid $M_1$. So $M_1|_F$ is supersolvable by \Cref{prop:inherit_ss}. Any full modular flag in $M_1|_F$ extends to a full modular flag in $M_1$ passing through $F$.
    Suppose instead that $F_1' \supsetneq F$. Let $M' \deq M|_H$. By \Cref{cor:modular_parts} and \Cref{prop:basic}\eqref{it:intersection}, $F_1'$ is modular in $M$; so is $F_0$ by \Cref{prop:basic}\eqref{it:restriction}. Thus $M'$ is the modular join of $F_0$ and $F_1'$, by \Cref{prop:ziegler}. If this non-trivial join is minimal, then by the inductive hypothesis either $M'|_{F_0} = M_0$ or $M|_{F_1'}$ has a full modular flag passing through $F$. In the first case, we win immediately. In the second case, the full modular flag in $M'|_{F_1'} = M|_{F_1'}$ extends to one in $M_0$ since the coatom $F_1'$ in $M_1$ is modular by \Cref{prop:basic}\eqref{it:restriction}. We now show that minimality must in fact hold. If not, then by \Cref{prop:magic} we can find a modular flat $G \subsetneq F$ of $M'$ such that $M'/G = M/G|_{H-G}$ is disconnected. By \Cref{prop:basic}\eqref{it:quotient}, $H-G$ is modular in $M/G$.
    Hence \Cref{prop:basic}\eqref{it:connected} forces $M/G$ to be disconnected. Using \Cref{prop:magic} once more, we see that $M$ is a non-trivial modular join over $G$. This contradicts the minimality of $F$.
\end{proof}
In the case of a parallel connection, the minimality condition can be dispensed with.
\begin{Cor}\label{cor:pc_ss}
    A \emph{parallel connection} $P_{f}(M', M'')$ is supersolvable if and only if both $M'$ and $M''$ are supersolvable and at least one of them has a full modular flag passing through $\{f\}$.
\end{Cor}
To prove this corollary, one observes that a parallel connection can only fail to be minimal in the above sense in a disconnected matroid, in which case we can argue component by component. We omit the details. In general, the minimality condition can be something of a nuisance. Luckily, the following gives a sufficient condition.
\begin{Lm}\label{lm:minimality_suff}
    Let $M = P_F(M_1, M_2)$ be a non-trivial modular join with with $F \neq \varnothing$. If $M_1$ and $M_2$ are both round, then the join is minimal in the sense of \Cref{thm:supersolvable_GPC}.
\end{Lm}
\begin{proof}
    Suppose instead that $M$ is a modular join over some (modular) flat $G \subsetneq F$. Then
    $M/G$ is disconnected by \Cref{prop:magic}. On the other hand, $M/G = P_{F - G}(M_1/G, M_2/G)$ by \Cref{lm:gpc_behaviour}\ref{it:contractify}.
    Since $M_1$ and $M_2$ are round, so are $M_1/G$ and $M_2/G$ by \Cref{prop:round_contract}. In particular, they are both connected. By \Cref{prop:gpc_connected}, $M/G$ must be connected; a contradiction.
\end{proof}
\begin{Ex}\label{rmk:error}
    In \cite{Z91}*{Theorem 3.4}, \Cref{thm:supersolvable_GPC} is stated incorrectly with the condition that \emph{both} $M_0$ and $M_1$ contain a full modular flag passing through $F$. 
    Given \Cref{thm:supersolvable_GPC}, to see that this condition is not necessary, it suffices to show that it is possible to construct a modular join as above in which only in $M_0$ is there a full modular flag passing through $F$. For this we can take a parallel connection of two connected matroids, which is automatically a minimal non-trivial modular join. For $M_0$ there are many choices, e.g.\ any simple matroid of rank 2 on $\geq 3$ elements. For $M_1$ it suffices to find a supersolvable matroid which does not admit full modular flags passing through all of its atoms.
    The signed-graphic matroid depicted in \Cref{fig:bad_matroid} is the smallest matroid with this property.  It is the matroid $\mathrm{O}_7$ in Oxley's catalogue \cite{Oxley}*{p.\ 644}.
    Its unique modular coatom is $\{4,5,6,7\}$.
    There are therefore no full modular flags passing through any of $1$, $2$ or $3$.
      \begin{figure}[htb!]
        \centering
    \begin{minipage}{0.33\textwidth}
        \centering
            \tikzset{
  LabelStyle/.style = { rectangle, rounded corners, draw,
                        minimum width = 2em},
  VertexStyle/.append style = { inner sep=5pt,
                                font = \Large\bfseries},
  EdgeStyle/.append style = {bend left} }
        \begin{tikzpicture}
        \tikzset{vertex/.style={circle,fill=blue!25,minimum size=12pt,inner sep=2pt}}
  \tikzset{every loop/.style={}}
        \node[vertex] (B) at (0,2)  [shape=circle,draw=black,fill=gray] {};
    \node[vertex] (C) at (0,-2)  [shape=circle,draw=black,fill=gray] {};
    \node[vertex] (D) at (4,0)  [shape=circle,draw=black,fill=gray] {};
    \path[loop/.style={looseness=30}]   (B) edge [loop above, thick] node {$5$} (B);
    \path[loop/.style={looseness=30}]   (C) edge [loop below,  thick] node {$4$} (C);
    \path[loop/.style={looseness=30}]   (D) edge [loop right,  thick] node {$1$} (D);
    \path (B) edge [bend left, thick] node[right]{$-$} node[left]{$7$} (C);
    \path (C) edge [bend left, thick] node[left]{$+$} node[right]{$6$} (B);
   \path (C) edge [ thick] node[below]{$+$} node[above]{$2$} (D); 
   \path (D) edge [thick] node[below]{$3$} node[above]{$+$} (B);
    \node[vertex] (D) at (4,0)  [shape=circle,draw=black,fill=gray] {};
\end{tikzpicture}
    \end{minipage}
    \begin{minipage}{0.6\textwidth}
        \centering
        \includestandalone[width=\linewidth]{lattice_of_flats}
    \end{minipage}
        \caption{A rank 3 supersolvable signed-graphic matroid and its lattice of flats.}
        \label{fig:bad_matroid}
    \end{figure}
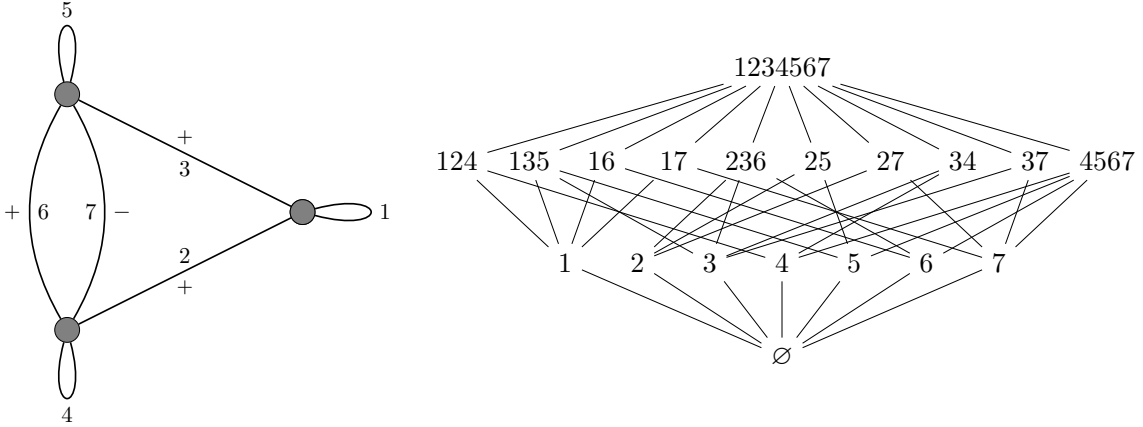
\end{Ex}

\begin{Ex}\label{ex:pcnonss} We can use the same building block to produce a matroid which is not supersolvable but is the parallel connection of two supersolvable ones. We let $M'$ and $M''$ be copies of the matroid in \Cref{fig:bad_matroid} on disjoint ground sets and define
$M \deq P_{1',1''}(M', M'')$, where 
$1', 1''$ are the atoms in each matroid corresponding to the element labeled $1$ in the diagram.
The matroid $M$ is a minimal modular join over the fused element $1' = 1''$, so fails to be supersolvable by \Cref{thm:supersolvable_GPC}. 
$M$ is again a signed-graphic matroid, defined by the signed graph in \Cref{fig:OSKoszul_nonSS_signed_graphic_matroid}; thus $M$ is realizable over $\mathbb{C}$ (and even $\mathbb{Q}$).  It corresponds to the complex hyperplane arrangement with hyperplanes defined by the following linear forms in $\mathbb{C}[x_1,x_2,y_1,y_2,z]$:
\[
\{z,x_1,x_2,x_1+x_2,x_1-x_2,x_1-z,x_2-z,y_1,y_2,y_1+y_2,y_1-y_2,y_1-z,y_2-z\}.
\]
\noindent
This arrangement was previously constructed by Tsujie \cite{Tsujie20}*{Example 4.12} as an example of a free arrangement whose intersection lattice is not supersolvable.
We show in \Cref{cor:OSparallel} that the Orlik--Solomon algebra of $M$ is Koszul.
\begin{figure}[ht!]
    \centering
    \tikzset{
  LabelStyle/.style = { rectangle, rounded corners, draw,
                        minimum width = 2em},
  VertexStyle/.append style = { inner sep=5pt,
                                font = \Large\bfseries},
  EdgeStyle/.append style = {bend left} }
        \begin{tikzpicture}
        \tikzset{vertex/.style={circle,fill=blue!25,minimum size=12pt,inner sep=2pt}}
  \tikzset{every loop/.style={}}
        \node[vertex] (B) at (0,2)  [shape=circle,draw=black,fill=gray] {};
    \node[vertex] (C) at (0,-2)  [shape=circle,draw=black,fill=gray] {};
    \node[vertex] (D) at (4,0)  [shape=circle,draw=black,fill=gray] {};
       \node[vertex] (E) at (8,2)  [shape=circle,draw=black,fill=gray] {};
    \node[vertex] (F) at (8,-2)  [shape=circle,draw=black,fill=gray] {};
    \path[loop/.style={looseness=30}]   (B) edge [loop above, thick] node {} (B);
    \path[loop/.style={looseness=30}]   (C) edge [loop below,  thick] node {} (C);
    \path[loop/.style={looseness=30}]   (D) edge [loop above,  thick] node {} (D);
      \path[loop/.style={looseness=30}]   (E) edge [loop above, thick] node {} (E);
    \path[loop/.style={looseness=30}]   (F) edge [loop below,  thick] node {} (F);
    \path (B) edge [bend left, thick] node[right]{$-$} node[left]{} (C);
    \path (C) edge [bend left, thick] node[left]{$+$} node[right]{} (B);
   \path (C) edge [ thick] node[below]{$+$} node[below]{} (D); 
   \path (D) edge [thick] node[above]{} node[above]{$+$} (B);
    \path (E) edge [bend left, thick] node[right]{$+$} node[left]{} (F);
    \path (F) edge [bend left, thick] node[left]{$-$} node[right]{} (E);
   \path (F) edge [ thick] node[below]{$+$} node[below]{} (D); 
   \path (D) edge [thick] node[above]{} node[above]{$+$} (E);
    \node[vertex] (D) at (4,0)  [shape=circle,draw=black,fill=gray] {};
\end{tikzpicture}
        \caption{A rank 5 OS-Koszul signed-graphic matroid which is not supersolvable.}
\label{fig:OSKoszul_nonSS_signed_graphic_matroid}
\end{figure}
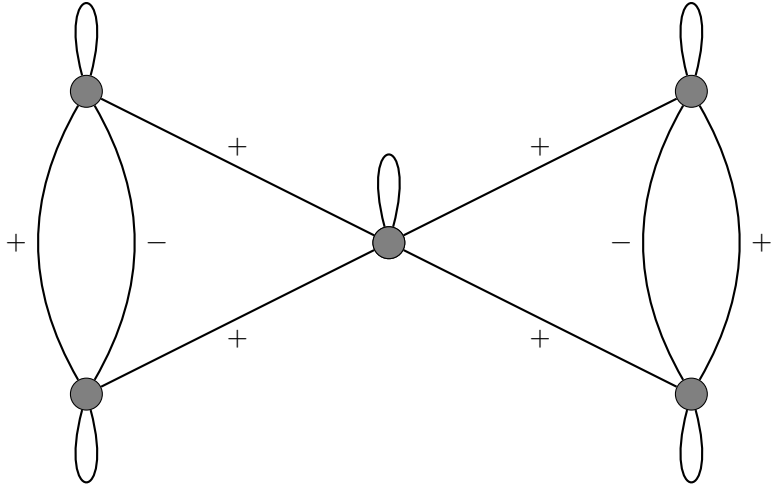
\end{Ex}

\section{Orlik--Solomon Algebras of Parallel Connections}\label{sec:parallel_connections}

In this section we prove that parallel connections of OS-Koszul matroids are OS-Koszul. In particular, this will show that the non-supersolvable matroid constructed in \Cref{ex:pcnonss} is OS-Koszul.
To motivate things, we start with a topological proof in the case of realizable arrangements.
We first explain how to realize parallel connections of realizable arrangements.

Suppose $\A$ and $\mathcal{B}$ are central hyperplane arrangements in $\CC^{r}$ and $\CC^{s}$ respectively.
 Write $x_1, \dots, x_r$ for the coordinates of $\CC^r$ and $y_1, \dots, y_s$ for the coordinates of $\CC^s$.
Choose hyperplanes $X = V(f) \in \A$ and $Y = V(g) \in \mathcal{B}$. Let $\mathcal{C} = \A \times \mathcal{B}$ be the product arrangement in $\CC^{r+s}$ and let $W$ be the hyperplane $V(f - g)$ in $\CC^{r+s}$ (using the coordinates $x_1, \dots, x_r, y_1, \dots, y_s$ on $\mathbb{C}^{r+s} = \mathbb{C}^r \times \mathbb{C}^s$).
Choosing a linear isomorphism $W \cong \CC^{r+s-1}$, the arrangement $P_{X, Y}(\A, \mathcal{B}) \deq \{W \cap H \mid H \in \A \times \mathcal{B}\}$ realizes the parallel connection $P_{X, Y}(M(\A), M(\mathcal{B}))$.
To see why, first note that by a change of coordinates we may assume that $f = x_r$ and $g = y_1$. In this case, the above description is a straightforward translation into the language of hyperplane arrangements of the construction in the proof of \cite{Oxley}*{Proposition 7.1.24}.

Since we have a diffeomorphism $\Cc(\A) \simeq \Cc(\mathbf{d}_H(\A)) \times \mathbb{C}^{\times}$ for any central arrangement $\A$ and hyperplane $H \in \A$, the following is a more precise form of \cite{falk_proudfoot}*{Corollary 4.3}, which, in the setting above, proves the diffeomorphism $\mathcal{C}(P_{X, Y}(\A, \mathcal{B})) \simeq \mathcal{C}(\mathbf{d}_{X}\mathcal{A}) \times \mathcal{C}(\mathcal{B})$.
\begin{Prop}\label[Prop]{prop:pc_complement}
    Let $\mathcal{A}_1, \mathcal{A}_2$ be central arrangements. Identify the symbol $H$ with a hyperplane in each of these arrangements. Let $\mathcal{A} = P_{H}(\mathcal{A}_1, \mathcal{A}_2)$ be the parallel connection arrangement as defined above. Then there exists a diffeomorphism
    $
        \mathcal{C}(\mathbf{d}_{H}\mathcal{A}) \simeq \mathcal{C}(\mathbf{d}_{H}\mathcal{A}_1) \times \mathcal{C}(\mathbf{d}_{H}\mathcal{A}_2)
    $.
\end{Prop}
\begin{proof}
Write $x_0, x_1, \dots, x_r$ for the coordinates of $\mathcal{A}_1$ and $y_0, \dots, y_s$ for the coordinates of $\mathcal{A}_2$. 
We may assume that $V(x_0)$ is the hyperplane $H$ in $\mathcal{A}_1$ and that $V(f_1), \dots, V(f_n)$ are its remaining hyperplanes.
Similarly, we may assume that $V(y_0)$ is the hyperplane $H$ in $\mathcal{A}_1$ and that $V(g_1), \dots, V(g_m)$ are its remaining hyperplanes.
The arrangement complements in question are affine varieties over $\mathbb{C}$. We compute their coordinate rings:
\begin{itemize}[leftmargin=0.75\leftmargini]
    \item For $\mathcal{C}(\mathcal{A}_1)$, it is the localization 
$\mathbb{C}[x_0, x_1, \dots, x_r]_{x_0f_1\cdots f_n}$.\\
    \hspace{1em}$\drsh$ For $\mathcal{C}(\mathbf{d}_H \mathcal{A}_1)$, it is  
$\mathbb{C}[x_0, \dots, x_r]_{x_0 f_1\cdots f_n}/(x_0-1)$.

    \item For $\mathcal{C}(\mathcal{A}_2)$, it is
$\CC[y_0, \dots, y_s]_{y_0 g_1\cdots g_m}$.\\
    \hspace{1em}$\drsh$ For $\Cc(\mathbf{d}_H \mathcal{A}_2)$, it is  
$\mathbb{C}[y_0, \dots, y_s]_{y_0 g_1\cdots g_m}/(y_0-1)$.

    \item For $\Cc(\A)$, it is  
$\mathbb{C}[x_0, \dots, x_r, y_0, \dots, y_s]_{x_0y_0f_1\cdots f_n g_1\cdots g_m}/(x_0 - y_0)$.\\
    \hspace{1em}$\drsh$ For $\Cc(\mathbf{d}_H\A)$, it is  
$\mathbb{C}[x_0, \dots, x_r, y_0, \dots, y_s]_{x_0y_0f_1\cdots f_n g_1\cdots g_m}/(x_0 - y_0, x_0 - 1)$.

    \item For $\Cc(\mathbf{d}_H\A_1) \times \Cc(\mathbf{d}_H\A_2)$, it is $\mathbb{C}[x_0, \dots, x_r]_{x_0 f_1 \cdots f_n}/(x_0-1) \otimes_{\mathbb{C}} \mathbb{C}[y_0,\dots, y_s]_{y_0 g_1 \cdots g_m}/(y_0-1)$. This ring is isomorphic to
    $\mathbb{C}[x_0, x_1, \dots, x_r, y_0, y_1, \dots, y_s]_{x_0 y_0 f_1\cdots f_r g_1\cdots g_s}/(x_0 - 1, y_0 - 1)$.
\end{itemize}
Since the ideals $(x_0 - y_0, x_0 - 1) = (y_0-1, x_0 - 1)$ coincide, $\Cc(\mathbf{d}_H \A)$ and $\Cc(\mathbf{d}_H\A_1) \times \Cc(\mathbf{d}_H\A_2)$ have isomorphic coordinate rings. They are thus isomorphic as varieties, hence diffeomorphic.
\end{proof}
\noindent
We get the following immediate consequence.
\begin{Cor}[{c.f.\ \cite{falk_proudfoot}*{Corollary 3.6}}]\label{cor:OSparallel}  Let $\A_1, \A_2$ and $H$ be as in \Cref{prop:pc_complement}.
The parallel connection $\A = P_H(\A_1, \A_2)$
is OS-Koszul if and only if both $\A_1$ and $\A_2$ are.
\end{Cor}
\begin{proof} 
By \Cref{prop:OSKoszul}, the OS-Koszul property for each of $\A_1, \A_2$, and $\A$ is equivalent to the same property for $\textbf{d}_H \A_1$, $\textbf{d}_H \A_2$, and $\textbf{d}_H \A$.
By \cite{PY99}*{Theorem 5.1}, the OS-Koszul property for each of these arrangements is equivalent to the corresponding arrangement complement being a rational $K(\pi, 1)$ space.
By \Cref{prop:pc_complement}, $\mathcal{C}(\mathbf{d}_{H} \mathcal{A})$ and $\mathcal{C}(\mathbf{d}_{H}\mathcal{A}_1) \times \mathcal{C}(\mathbf{d}_{H}\mathcal{A}_2)$ are homeomorphic. Taking $\QQ$-completions, we see from \cite{bousfield_kan}*{Lemma 7.2}
that $\QQ_{\infty}\mathcal{C}(\mathbf{d}_{H} \mathcal{A})$ is \emph{homotopic} to
$\QQ_{\infty}\mathcal{C}(\mathbf{d}_{H}\mathcal{A}_1) \times \QQ_{\infty}\mathcal{C}(\mathbf{d}_{H}\mathcal{A}_2)$. Taking homotopy groups, we get 
\[
    \pi_k(\QQ_{\infty}\mathcal{C}(\mathbf{d}_{H} \mathcal{A})) = \pi_k(\QQ_{\infty}\mathcal{C}(\mathbf{d}_{H}\mathcal{A}_1)) \times \pi_k(\QQ_{\infty}\mathcal{C}(\mathbf{d}_{H}\mathcal{A}_2))
\]
The left-hand side vanishes for all $k \geq 2$ if and only if $\A$ is a rational $K(\pi,1)$ arrangement. The right-hand side vanishes for all $k \geq 2$ if and only if both $\A_1$ and $\A_2$ are.
\end{proof}
\noindent
Applying K\"unneth's theorem \cite{Hatcher02}*{Theorem 3.15} in
\Cref{prop:pc_complement} gives $H^*(\mathcal{C}(\mathbf{d}_{H} \mathcal{A}), \QQ) = H^*(\mathcal{C}(\mathbf{d}_{H}\mathcal{A}_1), \QQ) \otimes_{\QQ} H^*(\mathcal{C}(\mathbf{d}_{H}\mathcal{A}_2), \QQ)$.
This translates to $\OS(\mathbf{d}_{H}\A) = \OS(\mathbf{d}_{H}\mathcal{A}_1) \otimes_{\QQ} \OS(\mathbf{d}_{H}\mathcal{A}_2)$. This gives a second proof of \Cref{cor:OSparallel}, by appealing to \Cref{prop:OSKoszul}.
By proving this statement about Orlik--Solomon algebras directly, one generalizes \Cref{cor:OSparallel} to arbitrary matroids.
\begin{Thm}[\cite{Falk01}*{Theorem 3.7}]\label[Thm]{prop:tensorreducedos}
    Let $M$ and $N$ be matroids and let us identify the symbol $f$ with an element in each.
    Then we have the graded $\QQ$-algebra isomorphism
    \[
        \rOS(P_{f}(M, N)) \cong \rOS(M)\otimes_{\QQ} \rOS(N).
    \]
    \begin{proof}
        Let $E_0 = E(M) - \{f\}$ and $E_1 = E(N) - \{f\}$.
        We will prove the isomorphism 
        $$\OS(\mathbf{d}_f P_{f}(M, N)) \cong \OS(\mathbf{d}_f M)\otimes_{\QQ} \OS(\mathbf{d}_f N).$$
        Both sides are canonically quotients of the free exterior algebra 
        $A = \bigwedge(E_0 \sqcup E_1)$.
        We will show that the kernel ideals are equal.
        As discussed in \Cref{subsec:OS_algebras},
        $\OS(\mathbf{d}_f P_{f}(M, N))$ is the quotient of $A$ by 
        (i) the monomials 
        $e_D$ for subets $D \subseteq E_0 \sqcup E_1$ containing $f$ in their span
        and (ii) the elements $\partial e_C$ for all circuits of $P_f(M, N)\del f$.
        Since a subset of $E_0$ (or $E_1$) spanning $f$ still spans it in $P_f(M, N)$ and since every circuit of $M \del f$ (or $N \del f$) is a circuit of $P_f(M, N) \del f$, every relation in the defining ideal of 
        $\rOS(M)\otimes_{\QQ} \rOS(N)$ lies in the defining ideal of
        $\OS(\mathbf{d}_f P_{f}(M, N))$. The converse follows from the following two observations:
        \begin{itemize}
            \item If $D \subseteq E_0 \sqcup E_1$ contains $f$ in its span in $P_f(M, N)$, then so does at least one of $D \cap E_0$ or $D \cap E_1$.
            \item If $C \subseteq E_0 \sqcup E_1$ is a circuit of $P_f(M, N)$ then either $C \subseteq E_0$, or 
            $C \subseteq E_1$, or else both 
            $C \cap E_0$ and $C \cap E_1$ contain $f$ in their span.
        \end{itemize}
        The first observation shows that every element of type (i) in the defining ideal of $\OS(\mathbf{d}_f P_{f}(M, N))$ is a multiple of an element of type (i) in the defining ideal of 
        $\rOS(M)\otimes_{\QQ} \rOS(N)$. 
        The second observation shows that every element of type (ii) in the defining ideal of $\OS(\mathbf{d}_f P_{f}(M, N))$ is either equal, up to sign, to a type (ii) element for $\rOS(M)\otimes_{\QQ} \rOS(N)$ or else every term in its expansion into monomials is a multiple of an element of type (i) in the defining ideal of $\rOS(M)\otimes_{\QQ} \rOS(N)$.
    \end{proof}
\end{Thm}
As promised, we now apply \Cref{prop:OSKoszul}
and \Cref{prop:tensorreducedos} to get the following as an immediate corollary, generalizing \Cref{cor:OSparallel}.
\begin{Cor}\label{cor:pc_Koszul}
    Let $M = P_f(M', M'')$ be a parallel connection. Then $M$ is OS-Koszul if and only if both $M'$ and $M''$ are OS-Koszul.
\end{Cor}
\noindent
Recall that a graded $\QQ$-algebra is \textbf{G-quadratic} if it has a quadratic Gr\"obner basis under \emph{some presentation}.
In the case of Orlik--Solomon algebras, this means that the Orlik--Solomon ideal has a quadratic Gr\"obner basis after applying a linear change of basis to the presenting exterior algebra. Note that the isomorphism $\OS(M) \cong \bigwedge {\mathbb{Q}^1} \otimes_\QQ\rOS(M)$ implies that G-quadraticity for the reduced and nonreduced Orlik--Solomon algebras are equivalent.
Since the tensor product of G-quadratic algebras is G-quadratic, \Cref{prop:tensorreducedos} implies:
\begin{Cor}\label[Cor]{prop:gquadratic}
    Let $M = P_f(M', M'')$ be a parallel connection.
    If $M'$ and $M''$ have G-quadratic (reduced) 
    Orlik--Solomon algebras then so does $M$.
    In particular, if $M$ and $N$ are supersolvable matroids, then for any $f\in M$, $g\in N$, both $\rOS(P_{f,g}(M, N))$ and $\OS(P_{f,g}(M, N))$ are G-quadratic.
\end{Cor}

\begin{Rmk}
  The special case of \Cref{prop:gquadratic} when $M$ and $N$ are supersolvable  can be seen even more directly. \Cref{prop:tensorreducedos} shows that the Orlik--Solomon algebra of a parallel connection is independent (as an abstract ring) of the choice of basepoints. Two supersolvable matroids always have \emph{some} supersolvable parallel connection: by \Cref{cor:pc_ss}
    we need only choose full modular flags in each one and parallel connect along the rank 1 flats in each.
In the case of a parallel connection of arrangements, this independence of the rings on the choice of basepoints reflects an even starker fact about the hyperplane arrangement complements. Namely, \Cref{prop:pc_complement} shows that the complement of the deconed parallel connection arrangement is independent up to diffeomorphism on the choice of basepoints!
\end{Rmk}
\noindent
It was for some time an open question whether the homotopy type of an arrangement $\A$ is determined by its intersection lattice, until Rybnikov \cite{rybnikov11} showed that 
$\pi_1(\Cc(\A))$ is not determined by $\mathcal{L}(\A)$.
Also considered in the literature is the converse question: to what extent can $\mathcal{L}(\A)$ be recovered from the homotopy type of $\Cc(\A)$?
Parallel connections give an easy way to produce arrangements with homotopic (in fact, diffeomorphic) complements but distinct underlying matroids.
\begin{Ex}\label{ex:not_iso}
    Let $\A'$ and $\A''$ be copies of the signed-graphic arrangement corresponding to the signed graph in \Cref{fig:bad_matroid}. For $1 \leq i \leq 7$, write $H_i'$ (resp.\ $H_i''$) for the hyperplane 
    in $\A'$ (resp.\ $\A''$) corresponding to the element labeled $i$ in the graph.  Let us consider three different parallel connection arrangements: $\A_{1,1} = P_{H_1',H_1''}(\A',\A'')$, $\A_{1,4} = P_{H_1',H_4''}(\A',\A'')$, and $\A_{4,4} = P_{H_4',H_4''}(\A',\A'')$.
    By \Cref{prop:pc_complement}, the complements 
    $\Cc(\A_{1,1})$, $\Cc(\A_{1,4})$, $\Cc(\A_{4,4})$ are diffeomorphic. (This is immediate for the deconed arrangements; for the central arrangements, it follows from the diffeomorphism $\Cc(\A) \simeq \Cc(\mathbf{d}_H \A) \times \mathbb{C}^{\times}$.)
    On the other hand,  it follows from \Cref{thm:supersolvable_GPC} that $\A_{1,4}$ and $\A_{4,4}$ are supersolvable while $\A_{1,1}$ is not. This implies in particular that $\mathcal{L}(\A_{1,1}) \ncong \mathcal{L}(\A_{1,4}),\, \mathcal{L}(\A_{4,4})$.
    The latter two are also distinct; see \Cref{rmk:not_erectible}.
\end{Ex}
\noindent
A variant of the construction in \Cref{ex:not_iso} is employed in \cite{EF99} to give \emph{reducible} arrangements with diffeomorphic complements; the authors go on to conjecture that the intersection lattices of \emph{irreducible arrangements} are determined by their homotopy type.
By \Cref{prop:gpc_connected}, \Cref{ex:not_iso} resolves this conjecture negatively.
Similar constructions produce examples of all ranks $\geq 5$. As a matter of fact, the earliest construction of this kind was given by Falk \cite{falk93}, who exhibits an example of homotopic complements in rank 3. 

Recall that the \textbf{truncation} of a matroid $M$ is the matroid $\tr M$ with the same ground set whose flats are those flats of $M$ which are not coatoms.
The truncation $\tr \A$ of a central arrangement $\A$ is obtained by first adding a generic new hyperplane $H$ to get a central arrangement $\A' = \A \cup \{H\}$ and then taking $\tr \A \deq (\A')^H$. Naturally, we have $M(\tr \A) = \tr M(\A)$.
Falk's arrangements may be reinterpreted as the truncations of two inequivalent parallel connections of two arrangements of ranks $2$ and $3$. The proof of homotopy equivalence here is much deeper, as evidenced by the fact that his arrangements have \emph{non}-homeomorphic complements by \cite{TanYau93}.

It is not clear that truncation will always preserve homotopy equivalence between two arrangement complements. However, it is promising that it \emph{does} preserve isomorphism between the corresponding Orlik--Solomon algebras, as shown by the following two lemmas.
\begin{Lm}\label{lm:equiv_iso}
    Let $\A$ be a central hyperplane arrangement.
    The isomorphism type of $\OS(\A)$ as a graded $\QQ$-algebra determines the isomorphism type of $\rOS(\A)$ and vice versa.
\end{Lm}
\begin{proof}
    The second claim follows from the graded isomorphism 
    $\OS(\A) \cong \rOS(\A) \otimes_{\QQ} \bigwedge \mathbb{Q}^1$.
    This isomorphism in turn shows that $\OS(\A)/(\ell) \cong \rOS(\A)$ for a \emph{generic} linear form $\ell \in \OS(\A)_1$.
\end{proof}
\begin{Lm}\label{lm:tr_iso}
    Let $\A$ be a central hyperplane arrangement. 
    The isomorphism type of $\rOS(\A)$ as a graded $\QQ$-algebra determines the isomorphism type of $\rOS(\tr \A)$.
\end{Lm}
\begin{proof}
    Suppose $M(\A)$ has rank $r$. Then the highest degree in which $\rOS(\A)$ does not vanish is $r-1$ and we have 
    $\rOS(\tr A) = \rOS(\A)/(\rOS(\A)_{r-1})$.
\end{proof}
\begin{Rmk}
    \Cref{lm:tr_iso} appears as \cite{Falk01}*{Theorem 3.11} but the proof there uses the unjustified assumption that isomorphic Orlik--Solomon algebras are isomorphic as \emph{differential-graded} algebras.
    An earlier claimed proof in \cite{Pendergrass98} suffers from the same flaw.
\end{Rmk}
\noindent
A matroid $M$ is \textbf{inerectible} if it is not the truncation of another matroid on the same ground set; an \textbf{erection} of $M$ is a matroid $N$ with $\tr N = M$.
Falk and Randell asked in \cite{FR00}*{Problem 1.5} for irreducible arrangements with non-isomorphic \emph{inerectible} underlying matroids and homotopic complements.
This question is motivated by the above observation about Orlik--Solomon algebras and deliberately excludes Falk's construction as well as the reducible arrangements mentioned earlier.
The following lemma will show that \Cref{ex:not_iso} fulfills the Falk--Randell criteria.
\begin{Lm}\label{lm:pcinerect}
    A non-trivial parallel connection of connected simple matroids is connected and inerectible.
\end{Lm}
We will need the following theorem of Crapo. To state it, we shall convene that any matroid $M$ is a \emph{trivial} erection of 
itself. Inerectible matroids are those with only the trivial erection.
\begin{Thm}[{\cite{C70}*{Theorem 2}}]\label{thm:crapoerect}
    A collection $\Fc$ of subsets (called ``blocks'') of the ground set of a matroid $M$ of rank $r$ is the set of rank $r$ flats of an erection of $M$ if and only if 
    \begin{enumerate}
        \item each block spans $M,$
        \item each block is $(r-1)$-closed,
        \item each basis of $M$ is contained in a unique block.
    \end{enumerate}
Here, a subset $S\subseteq M$ is said to be $k$-closed if the closure of any $k$-element subset of $S$ lies in $S$.
\end{Thm}
\begin{proof}[Proof of \Cref{lm:pcinerect}]
    Let $M_1, M_2$ be two connected matroids of rank $\geq 2$ and let $M = P_f(M_1, M_2)$.
    The connectedness of $M$ follows from \Cref{prop:gpc_connected}. We will show inerectibility.
    Suppose $N$ is an erection of $M$.
    Let $B_1, B_2$ be bases of $M_1, M_2$, respectively, each containing $f$. Then $B = B_1\cup B_2$ is a basis of $M.$ Let $F$ be a block in $N$ containing $B$. 
    Since the connection is non-trivial, $B_1$ is a proper subset of $B$.
    So by condition (2) in \Cref{thm:crapoerect}, $F$ contains the span of $B_1$ which is $E(M_1)$.
    Similarly, $F$ must contain $E(M_2)$, so $F = E(M)$. This erection is trivial.
\end{proof}

\begin{Rmk}\label{rmk:not_erectible}
    We describe another way to see that the matroids in \Cref{ex:not_iso} are inerectible, which incidentally shows that they are pairwise non-isomorphic.  Let $W_k(M)$ denote the number of rank $k$ flats of a matroid $M$; these are the \textbf{Whitney numbers of $M$ of the second kind}.  By a simple computation, the sequences of Whitney numbers for the three matroids above are as follows:\\
    \begin{center}
    \begin{tabular}{c|cccccc}
    Matroid & $W_0$ & $W_1$ & $W_2$ & $W_3$ & $W_4$ & $W_5$\\
    \hline
    $M_{1,1}$ & $1$&  $13$& $56$& $90$& $44$& $1$\\
    $M_{1,4}$ & $1$&  $13$& $56$& $92$& $49$& $1$\\
    $M_{4,4}$ & $1$&  $13$& $56$& $95$& $55$& $1$\\
    \end{tabular}
    \end{center}

    \vspace{5mm}
    \noindent If one of these matroids were erectible, it would be the truncation of a matroid $\widetilde{M}$ of rank $6$. However, the data would then imply that $W_2(\widetilde{M}) > W_4(\widetilde{M})$, contradicting the Top Heavy Theorem \cite{singular:Hodge:theory:for:combinatorial:geometries}*{Theorem 1.1}.  Therefore, none of these matroids are erectible.
    The above discussion has an interesting corollary: while all three Orlik--Solomon algebras are isomorphic as graded algebras, each is naturally graded by a different lattice.
\end{Rmk}
\begin{Rmk}
Falk--Proudfoot's result stated here as \Cref{thm:FP}
is the ``if'' direction of our \Cref{thm:main} in the special case where $M$ is realizable over $\CC$. 
Under this assumption, the argument in \cite{falk_proudfoot} adapts easily to show that if both $\A$ and $\A_X$ have Koszul Orlik--Solomon algebras then the same holds for $\A_{\pi}$, though the authors did not point this out. These implications suffice for our three mechanisms of producing Koszul Orlik--Solomon algebras from non-supersolvable arrangements. 
This fact contrasts with the pessimism of Falk and Proudfoot, who express the view that their result cannot be used to produce a negative answer to \Cref{qst:mainquestion}.
This appears to have been partly due to their (incorrect) assumption that ``if $\A_X$ and $\A_{\pi}$ are supersolvable, then $\A$ is also supersolvable.'' In fact, this assumption was \emph{almost} correct, as shown by \Cref{cor:truncation_not_ss}.
\end{Rmk}
\begin{Rmk}
    It is worth noting that the
    analogue of \Cref{cor:pc_Koszul}
    does not hold for the related graded M\"obius algebra.  The graded M\"obius algebra $\GMA{M}$ of a simple matroid $M$ is a commutative algebra with defining equations determined by the circuits of the lattice of flats of $M$, just as with the Orlik--Solomon algebra.  In the realizable case, $\GMA{M}$ defines the cohomology ring of the matroid Schubert variety of $M$ studied by Ardila and Boocher \cite{Ardila-Boocher:Schubert:varieties}; see \cite{HW17}*{Theorem 14} and \cite{singular:Hodge:theory:for:combinatorial:geometries}*{Section 1.3} for precise statements.  These were key components in the proof of the Dowling--Wilson 
    conjectures
    by Braden, Huh, Matherne, Proudfoot, and Wang \cite{singular:Hodge:theory:for:combinatorial:geometries}; it had previously been proved in the realizable case by Huh and Wang \cite{HW17}.

    The algebraic properties of graded M\"obius algebras were studied in \cite{LMMP24}.  In particular, Theorem A of that paper shows that if $M(G)$ is the graphic matroid associated to a simple graph $G$, then its graded M\"obius algebra $\GMA{M(G)}$ is Koszul if and only if $G$ is strongly chordal.  The simplest chordal graph which is not strongly chordal is the tent graph $T$ depicted in Figure~\ref{fig:tent} (also known as the $3$-sun or $3$-trampoline).  It is the parallel connection of a $3$-cycle $C$ and a graph $B$ termed the broken $3$-trampoline in \cite{LMMP24}. 
    The connection takes place along the edges labeled $e_1$ and $e_2$ in the figure.
    Both $\GMA{M(B)}$ and $\GMA{M(C)}$ are Koszul, since $M(B)$ and $M(C)$ are both strongly chordal. Thus the parallel connection of matroids with Koszul Graded M\"obius algebras need not be Koszul.
    
    On the other hand, the parallel connection of $C$ and $B$ along the edges labeled $e_1$ and $e_3$ produces a strongly chordal graph $G$ (not pictured) so that $\GMA{M(G)}$ is Koszul; in particular, this means that $\GMA{M(G)} \not\cong \GMA{M(T)}$.  By \Cref{prop:tensorreducedos}, $\OS(M(G)) \cong \OS(M(T))$.
    Since $G$ and $T$ are both chordal graphs, $M(G)$ and 
    $M(T)$ are supersolvable by \cite{supersolvable}*{Proposition 2.8}. So \emph{both} their Orlik--Solomon algebras are Koszul.

    \begin{figure}[hbt!]
        \centering
\newcommand{\edgelabeledgraph}[3][scale = 1.2]{
	\begin{tikzpicture}[#1]          
	\newcommand*\points{#2}     
	\newcommand*\edges{#3}          
	\newcommand*\scale{0.75}          
	\foreach \x/\y/\z/\w/\a in \points {
		\draw[fill = black!50] (\scale*\x,\scale*\y) circle [radius = 0.1] node[label = {[label distance = 0.05 cm]\a: $\w$}] (\z) {}; 
	}
	\foreach \x/\y/\p/\a/\l in \edges { \draw (\x) -- (\y) node [pos = \p, sloped, \a] {{\small $\l$}}; }      
	\end{tikzpicture}
}

\edgelabeledgraph{
    -2/0/1//270,
    0/0/2//270,
    2/0/3//270,
    -1/1.73/4//270,
    1/1.73/5//120,
    -1/2.5/6//90,
    1/2.5/7//45,
    0/4.23/8//90,
    5/0/9//0,
    7/0/10//0,
    9/0/11//0,
    6/1.73/12//0,
    8/1.73/13//0,
    7/3.46/14//0
    }{
    1/2/0.5/above/, 
    2/3/0.5/below/, 
    1/4/0.5/below/e_3,
    4/2/0.5/below/,
    4/5/0.5/below/e_2,
    2/5/0.5/above/,
    3/5/0.5/below/,
    6/7/.5/above/e_1,
    6/8/.5/below/,
    7/8/.5/below/,
    9/10/.5/below/,
    10/11/.5/below/,
    9/12/.5/below/,
    10/12/.5/below/,
    12/13/.5/below/e,
    10/13/.5/below/,
    11/13/.5/below/,
    12/14/.5/below/,
    13/14/.5/below/}
        \caption{The tent graph as a parallel connection of two strongly chordal graphs}
        \label{fig:tent}
    \end{figure}
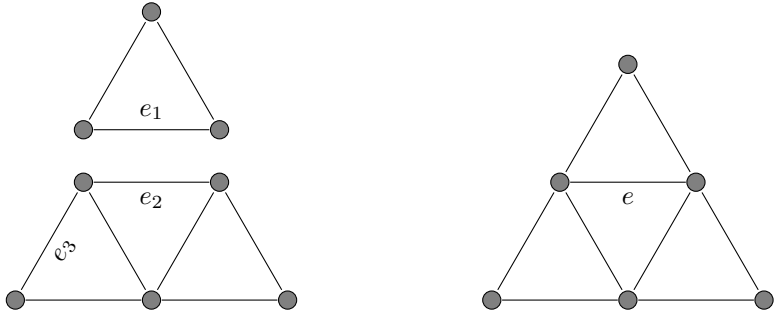
\end{Rmk}

\section{Koszul Orlik--Terao Algebras}\label{sec:koszul:OT}

Let $\A = \{H_1,\ldots,H_n\}$ be a complex hyperplane arrangement in $\CC^d$.  Write $S = \CC[x_1,\ldots,x_d]$ and $H_i = V(\ell_i)$ for $1 \le i \le d$, where each $\ell_i \in S$ is a homogeneous linear form.  Its Orlik--Terao algebras, introduced by Orlik and Terao \cite{OrlikTeraoAlgebras}, \cite{TeraoReciprocals},  are commutative counterparts to the Orlik--Solomon algebra.  Unlike the Orlik--Solomon algebra, the defining equations of the Orlik--Terao algebra do not depend solely on the intersection lattice of the arrangement.

The $\CC$-algebra $\OT(\A) := \CC\left[\ell_1^{-1},\ell_2^{-1},\ldots,\ell_n^{-1}\right] \subseteq 
\CC(x_1,\ldots,x_d)$ is the \textbf{Orlik--Terao algebra of $\A$}.  Equivalently, one can define $\OT(\A) = \CC[f_1,\ldots,f_n]$, where $f_i = \ell_1\cdots\widehat{\ell_i} \cdots \ell_n$.  Clearly $\OT(\A)$ has the structure of a standard graded integral domain over $\CC$.  
    Let $T = \CC[y_1,\ldots,y_n]$.
Define the surjective $\CC$-algebra homomorphism
    \[\Phi:T \to \CC\left[\frac{1}{\ell_1},\frac{1}{\ell_2},\ldots,\frac{1}{\ell_n}\right]\]
    by setting $\Phi(y_i) = \ell_i^{-1}$ for $1 \le i \le n$.  To make this a graded map, we set $\deg(x_i) = -1$ for $1 \le i \le d$.   Then $I_{\OT(\mathcal{A})} := \ker(\Phi)$ is  a graded prime ideal of $T$, called the \textbf{Orlik--Terao ideal} of $\A$. Terao \cite{TeraoReciprocals}*{Theorem 1.1} showed that 
    \[I_{\OT(\mathcal{A})} = \left( \sum_{i = 1}^t c_i y_{j_1}\cdots\widehat{y_{j_i}}\cdots y_{j_t} \;\bigg|\; \sum_{i = 1}^t c_i \ell_{j_i} = 0\right)\]
so that $\OT(\mathcal{A}) \cong T/I_{\OT(\mathcal{A})}$; see \cite{Berget}*{Theorem 4.3} for a characteristic-free proof.  Moreover, he showed that $\HS_{\OT(\mathcal{A})}(t) = \pi(\mathcal{A},t(1-t)^{-1})$ \cite{TeraoReciprocals}*{Theorem 1.2}.  In fact, Proudfoot and Speyer \cite{ProudfootSpeyerBrokenCircuitRing} later showed that the terms coming from circuits form a universal Gr\"obner basis of $I_{\OT(\mathcal{A})}$ and that $\OT(\mathcal{A})$ is Cohen-Macaulay.

The \emph{Artinian Orlik--Terao algebra of $\mathcal{A}$}, denoted $\AOT(\mathcal{A})$, is the following quotient of $\OT(\mathcal{A})$:
\[ \AOT(\mathcal{A}) := \OT(\mathcal{A})/(y_1^2,y_2^2,\ldots,y_n^2).\]
We set $I_{\AOT(\mathcal{A})} = I_{\OT(\mathcal{A})} + (y_1^2,\ldots,y_n^2) \subseteq T$, so that $\AOT(\mathcal{A}) \cong T/I_{\AOT(\mathcal{A})}$.
Orlik and Terao \cite{OrlikTeraoAlgebras} used the Artinian Orlik--Terao algebra to resolve a conjecture of Aomoto \cite{Aomoto96}. 
In the same paper, they show that, just as for the Orlik--Solomon algebra, we have $\HS_{\AOT(\mathcal{A})}(t) = \pi(\mathcal{A},t)$.   

We have the following implications for Orlik--Terao algebras:
\begin{align*}
    \mathcal{L}_\mathcal{A} \text{ is supersolvable } & \Longleftrightarrow I_{\OT(\mathcal{A})} \text{ has a quadratic Gr\"obner basis}\\
    & \Longrightarrow \OT(\mathcal{A}) \text{ is Koszul } \\
    & \Longrightarrow \OT(\mathcal{A}) \text{ is quadratic.}
\end{align*}
That supersolvable arrangements have Orlik--Terao ideals with quadratic Gr\"obner bases was shown in \cite{Denham14}.  
The converse follows along the same lines as \cite{Peeva}*{Theorem 4.3}
using \cite{ProudfootSpeyerBrokenCircuitRing}*{Theorem 4}.
The last implication in the sequence is strict, and the authors of \cite{Denham14} ask about the converse to the remaining implication.  We show below that it is also strict. 

Just as with Orlik--Solomon algebras, the Koszul property of Orlik--Terao algebras passes to parallel connections of arrangements.

\begin{Thm}\label{thm:pc_OT_Koszul}
    Suppose $\mathcal{A}$ and $\mathcal{A}'$ are central complex hyperplane arrangements.  Let $\mathcal{B}$ denote the parallel connection of $\A$ and $\A'$ obtained by identifying a hyperplane in one with a hyperplane in the other.
If $\OT(\A)$ and $\OT(\A')$ are Koszul, then $\OT(\mathcal{B})$ is Koszul.
\end{Thm}

\begin{proof}[Proof of \Cref{thm:pc_OT_Koszul}]
Since $\OT(\A)$ and $\OT(\A')$ are both Koszul,
\Cref{thm:tensor_product_Koszul} implies the same for their tensor product,
which coincides with $\OT(\A \times \A')$.
Write $\OT(\A) = \CC[y_1,\ldots,y_n]/I_{\OT(\A)}$ and $\OT(\A') = \CC[z_1,\ldots,z_m]/I_{\OT(\A')}$.  Set $T = \CC[y_1,\ldots,y_n,z_1,\ldots,z_m]$.  Let $y_1, z_1$ be the variables corresponding to the identified hyperplanes from $\A$ and $\A'$, respectively.  
By \cite{SchenckTohaneanuOTAlgebras}*{Lemma 3.2},  the defining ideals of the Orlik--Terao algebras are generated by the elements corresponding to chordless circuits\footnote{A circuit $C$ in a matroid $M$ is \emph{chordless} if it cannot be written as the symmetric difference $C_1 \Delta C_2$ where $C_1, C_2$ are circuits of $M$ with 
$\abs{C_1 \cap C_2} = 1$.}. Since the chordless 
circuits of a parallel connection of two matroids are just the chordless circuits of one or the other matroid, it follows that
\[
    \OT(\mathcal{B}) \cong T/(I_{\OT(\A)}T + I_{\OT(\A')}T + (y_1-z_1)) \cong \OT(\A \times \A')/(y_1-z_1).
\]
Since $\OT(\A \times \A')$ is an integral domain, $y_1 - z_1$ is a nonzerodivisor on $\OT(\A \times \A')$.  By \cite{BF85}*{Theorem 4(e)(iv)}, this implies that 
$\OT(\mathcal{B})$ is Koszul.
\end{proof}

To prove the analogous result for the Artinian Orlik--Terao algebra, we recall the following notion from \cite{HS11}.  Let $R$ be a standard graded $\CC$-algebra.  Two homogeneous elements $a,b \in R$ are said to be an \textbf{exact pair of zerodivisors} if 
$\ann_R a = b$ and $\ann_R b = a$, where $\ann_R$ denotes the annihilator of an element.

\begin{Thm}\label{thm:pc_AOT_Koszul}
    Suppose $\mathcal{A}$ and $\mathcal{A}'$ are central  complex hyperplane arrangements.  Let $\mathcal{B}$ denote the parallel connection of $\A$ and $\A'$ obtained by identifying a hyperplane in one with a hyperplane in the other.
If $\AOT(\A)$ and $\AOT(\A')$ are Koszul, then $\AOT(\mathcal{B})$ is Koszul.
\end{Thm}

\begin{proof}
    By \Cref{thm:tensor_product_Koszul},
    $\AOT(\A \times \A') = \AOT(\A) \otimes_{\CC} \AOT(\A')$ is Koszul,
    since both $\AOT(\A)$ and $\AOT(\A')$ are. Write $\AOT(\A) = \CC[y_1,\ldots,y_n]/I_{\AOT(\A)}$ and $\AOT(\A') = \CC[z_1,\ldots,z_m]/I_{\AOT(A')}$, set $T = \CC[y_1,\ldots,y_n,z_1,\ldots,z_m]$, and let $y_1,z_1$ be the variables corresponding to the identified hyperplanes.  Then 
    $$\AOT(\mathcal{B}) \cong T/(I_{\AOT(\A)}T + I_{\AOT(\A')}T + (y_1-z_1)) \cong \AOT(\A \times \A')/(y_1-z_1).$$

    It is easy to see that
    \[\AOT(\A \times \A')/(y_1-z_1) \cong \AOT(\A \times \A')/(y_1+z_1)  \]
    via the linear isomorphism preserving the $y_i$ and mapping $z_i \mapsto -z_i$ for all $i$.  Therefore, their Hilbert series agree.
    Write $\pi(\A, t) = (1 + t)g(t)$ and 
    $\pi(\A', t) = (1 + t)h(t)$.
    Then $\pi(\A\times \A', t) = (1 + t)^2 g(t)h(t)$.
    It follows from \Cref{prop:tensorreducedos}
    that $\pi(\mathcal{B}, t) = (1 + t)g(t)h(t)$. Thus
    \[
        \HS_{\AOT(\A \times \A')/(y_1+z_1)}(t) = \HS_{\AOT(\A \times \A')/(y_1-z_1)}(t) = 
        \HS_{\AOT(\mathcal{B})}(t) = (1 + t)g(t)h(t).
    \]
    From the graded short exact sequence
    \[
        0 \to \AOT(\A \times \A')/\ann(y_1-z_1)(-1) \to \AOT(\A \times \A') \to \AOT(\A \times \A')/(y_1-z_1) \to 0
    \]
    we get 
    \begin{align*}
    t\HS_{\AOT(\A \times \A')/\ann(y_1-z_1))}(t)
    &=
    \HS_{\AOT(\A \times \A')}(t) - \HS_{\AOT(\A \times \A')/(y_1-z_1))}(t)\\
    &= (1+t)^2g(t)h(t) - (1+t)g(t)h(t)\\
    &= t(1+t)g(t)h(t).
    \end{align*}
    Hence
    $\HS_{\AOT(\A \times \A')/\ann(y_1-z_1)}(t) = (1+t)g(t)h(t) = \HS_{\AOT(\A \times \A')/(y_1+z_1)}(t)$.  Because $(y_1+z_1) \subseteq \ann(y_1-z_1)$ and both quotients have the same Hilbert series, these ideals are equal.  An identical argument shows that 
    $(y_1-z_1) = \ann (y_1 + z_1)$.
    Thus $y_1-z_1$, $y_1+z_1$ are an exact pair of zero divisors on $\AOT(\A \times \A')$.
    The Koszulness of $\AOT(\mathcal{B})$ now follows from \cite{HS11}*{Corollary 1.12}, using the fact that $\AOT(\A \times \A')$ is Koszul.
\end{proof}

\begin{Cor}
    For any rank $r \ge 5$, there exist central, irreducible,  complex hyperplane arrangements such that $\OT(\A)$ and $\AOT(\A)$ are Koszul whereas the intersection lattice $\mathcal{L}(\A)$ is not {supersolvable}; equivalently, the defining ideals of $\OT(\A)$ and $\AOT(\A)$ do not have a quadratic Gr\"obner bases with respect to any monomial order in the presenting polynomial ring.
\end{Cor}
\begin{proof}
    By \Cref{lm:Dow_construction} we can find a connected submatroid $M$ (resp.\ $N$) of rank $r-2$ (resp.\ $2$) of
    the Dowling geometry $MQ_{r-2}(2)$ (resp.\ $MQ_3(2)$) containing the element $x_1$ such that this element does not lie in any modular line of $M$ (resp.\ $N$). 
    The parallel connection $P_{x_1}(M, N)$ has rank $r$ and is connected by \Cref{prop:gpc_connected}. It fails to be supersolvable by \Cref{thm:supersolvable_GPC}.
    The corresponding central arrangement will do by Theorems \ref{thm:pc_OT_Koszul} and \ref{thm:pc_AOT_Koszul}.
\end{proof}

\begin{Rmk} 
Suppose $\A$ is a central hyperplane arrangement in $\mathbb{R}^d$.  The \textbf{Varchenko--Gel'fand ring} of $\mathcal{\A}$, denoted $VG_{\QQ}(\A)$, is the ring of locally-constant $\QQ$-valued functions on the \emph{real complement} of $\A$, i.e.\ $\Cc(\A) \cap \mathbb{R}^d$.
Gel'fand and Varchenko equip this ring with a filtration;
the associated graded ring is the \textbf{graded Varchenko--Gel'fand ring}, denoted $\mathcal{V}_\QQ(\A)$. These were introduced in \cite{VG87}.  In \cite{Dorpalen-Barry23}*{Theorem 2}, Dorpalen-Barry showed that $\mathcal{V}_\QQ(\A)$ is Koszul if $\A$ is supersolvable and asks if the converse holds.  An argument nearly identical argument to the proof of \Cref{thm:pc_AOT_Koszul} shows that the converse fails.  We omit the details.
\end{Rmk}

\section{Generalized parallel connections}\label{sec:GPC_Koszul}
In this section, we use \Cref{thm:main} to show that generalized parallel connections preserve the OS-Koszul property and use the results of \Cref{sec:modular_joins} and \Cref{sec:parallel_connections} to show that many such examples can be constructed which are not supersolvable. 
To show that our construction produces $\mathbb{Q}$-realizable matroids, we provide a realizability result for generalized parallel connections of submatroids of Dowling geometries along Dowling subgeometries.

\begin{Lm}[{\cite{constructions}*{Lemma 7.7}}]\label{lm:truncation_computation}
    Let $M = P_{F}(M_1, M_2)$ be a generalized parallel connection.
    Then the matroids $\si(\overline{T}_{F}(M_1))$ and $\si(\overline{T}_{E(M_2)}(M))$ are isomorphic.
\end{Lm}
\begin{proof}[{Proof of \Cref{thm:GPC_Koszul}}]
Let $M = P_{F}(M_1, M_2)$ be a generalized parallel connection.
Here $F$ is a modular flat of $M_1$ and $E(M_2)$ is a modular flat of $M$ by \Cref{prop:side_modular}.
We assume that $M_1$ and $M_2$ are OS-Koszul.
By \Cref{thm:main}, to check that $M$ is OS-Koszul, it suffices to check that $\overline{T}_{E(M_2)}(M)$ is Koszul.
It follows from \Cref{lm:truncation_computation} that this matroid has the same Orlik--Solomon algebra as 
$\overline{T}_{F}(M_1)$. 
By another application of \Cref{thm:main}, $\overline{T}_{F}(M_1)$ is OS-Koszul. 
\end{proof}
\noindent
    Tsujie  \cite{Tsujie20} defined the class of \emph{modularly extended} matroids to consist of matroids obtainable from supersolvable matroids by recursively taking modular join along round flats. He showed that hyperplane arrangements whose underlying matroids are modularly extended are free, but that these matroids need not be supersolvable. Since modular joins along round flats are a special case of generalized parallel connections, \Cref{thm:GPC_Koszul} has the following corollary.
\begin{Cor}\label{cor:tsujie}
    Modularly extended matroids are OS-Koszul.
\end{Cor}
\noindent
There is a discrepancy between \Cref{thm:GPC_Koszul} and the analogous result for supersolvability: a generalized parallel connection of supersolvable matroids need not be supersolvable. 
As noted earlier, this was already noticed in \cite{Tsujie20}.
We make this failure more salient by exhibiting a large number of non-supersolvable modular joins coming from submatroids of cyclic Dowling geometries.
Before doing so, it will be convenient to note that our construction will stay inside this class.

\begin{Prop}\label{prop:modular_join_rep}
Let $m, r_1, r_2, s$ be positive integers. Let $M_1$ (resp.\ $M_2$) be a submatroid of 
$MQ_{r_1}(m)$ (resp.\ $MQ_{r_2}(m)$) containing a rank-$s$ Dowling subgeometry as a flat $F$.
Then the modular join $P_{F}(M_1, M_2)$ is a submatroid of $MQ_{r}(m)$, where 
$r \deq r_1 + r_2 - s$.
\end{Prop}
\begin{proof}
    We start with the case 
    $M_1 = MQ_{r_1}(m)$, $M_2 = MQ_{r_2}(m)$.
    We may embed these in $MQ_{r}(m)$ by identifying $M_1$ with the Dowling subgeometry
    $MQ_{\{1,\dots,r_1\}}(m)$, $M_2$ with 
    $MQ_{\{r_1-s+1,\dots,r\}}(m)$ and $F$ with $MQ_{\{r_1-s+1,\dots,r_1\}}(m)$.
    Let $M_0$ be the restriction of $MQ_{r}(m)$ to the union of their ground sets.
    The subsets $E(M_1)$ and $E(M_2)$ are modular in $MQ_r(m)$ by \Cref{prop:modular_in_Dow}, hence also in $M_0$ by \Cref{prop:basic}\eqref{it:deletion}.
    By \Cref{prop:ziegler}, $M_0$ is the modular join of $M_1$ and $M_2$.

    In the general case, we can embed $M_1$ as a \emph{submatroid} of $MQ_{\{1,\dots,r_1\}}(m)$
    and $M_2$ as a \emph{submatroid} of $MQ_{\{r_1-s+1,\dots,r\}}(m)$ so that $F$ is identified in both with the Dowling subgeometry $MQ_{\{r_1-s+1,\dots,r_1\}}(m)$.
    Let $M$ be the restriction of $MQ_r(m)$ to the union of their ground sets.
    Then with $M_0$ as above, we have by \Cref{lm:gpc_behaviour}\ref{it:restrictify},
    \[
        M = M_0|_{E(M)} = P_{F}(MQ_{\{1,\dots,r_1\}}(m), MQ_{\{r_1-s+1,\dots,r\}}(m))|_{E(M)} = 
        P_F(M_1, M_2). \qedhere
    \]
\end{proof}

\begin{Cor}\label{cor:realizability}
    The modular join of two submatroids of cyclic Dowling geometries (resp.\ two simple signed graphic arrangements) along a common Dowling subgeometry is 
    again a submatroid of a cyclic Dowling geometry (resp.\ a simple signed-graphic arrangement); in particular it is 
    $\mathbb{C}$-realizable (resp.\ $\mathbb{Q}$-realizable).
\end{Cor}

\begin{Rmk}
    The care we take in \Cref{cor:realizability} to establish realizability of the generalized parallel connection is not in vain: \cite{Oxley}*{11.4.17} shows that modular joins of realizable matroids need not be realizable.
\end{Rmk}
\noindent
We now have a very general procedure to produce $\QQ$-realizable modular joins which are OS-Koszul but not supersolvable.
Let $r_1, r_2, k$ be positive integers with 
$r_1, r_2 \geq k + 2$. By \Cref{lm:Dow_construction}, we can find round, supersolvable matroids $M_1$ of rank $r_1$ and $M_2$ of rank $r_2$ as submatroids of Dowling geometries of the same rank such that each contains a modular rank-$k$ Dowling subgeometry as a modular flat $F$ not contained in any full modular flag. By \Cref{lm:minimality_suff}, we may apply \Cref{thm:supersolvable_GPC} to deduce that the modular join $P_F(M_1, M_2)$ is not supersolvable. By \Cref{thm:GPC_Koszul}, it \emph{is} OS-Koszul.
Its rank is $r_1 + r_2 - k$.
It is connected by \Cref{prop:gpc_connected}.
If we choose the order of the Dowling geometries to be $2$, then $M$ will be a signed-graphic matroid by \Cref{cor:realizability}, hence $\QQ$-realizable.

\begin{Ex}\label{ex:non_ss_modular_join}
Consider the matroid shown in \Cref{fig:another_signed_graphic_matroid}, which is obtained as the modular join of two copies of the matroid in \Cref{fig:signed_graphic_matroid} along the Dowling subgeometry in bold. By the above discussion, this is a $\QQ$-realizable, non-supersolvable, connected OS-Koszul matroid.
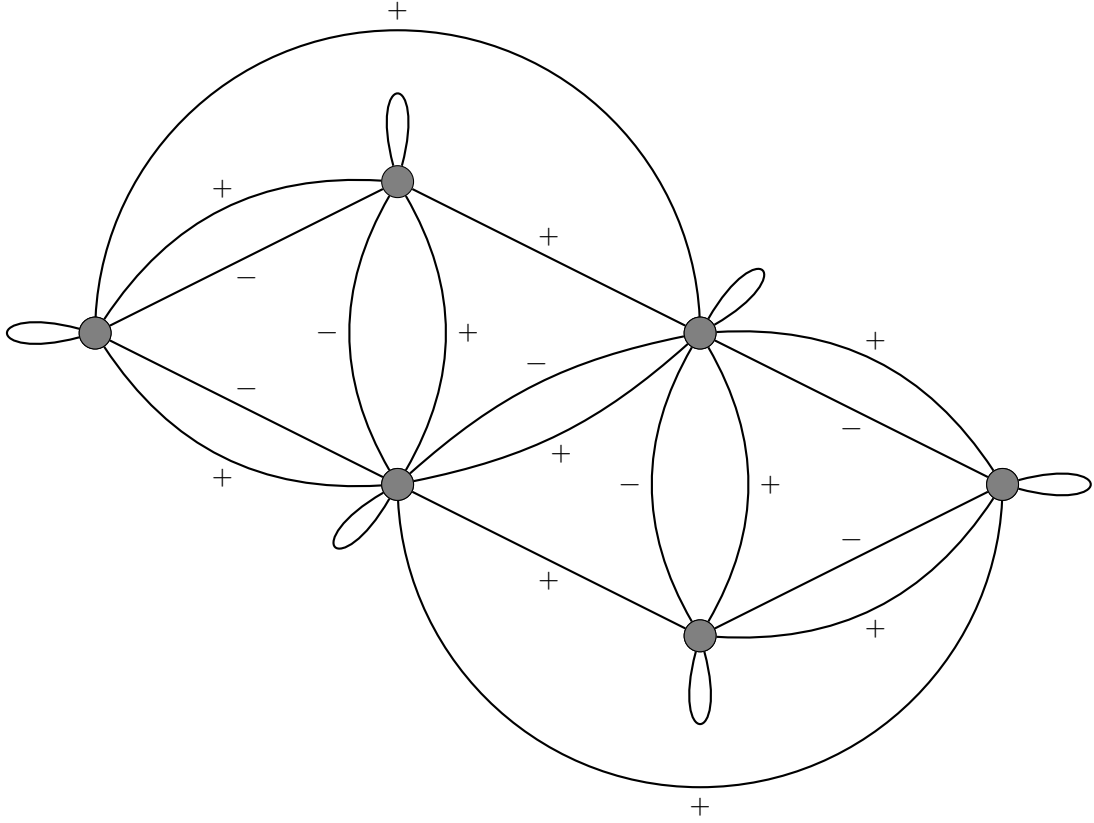
\begin{figure}[ht]
        \centering
                \tikzset{
  LabelStyle/.style = { rectangle, rounded corners, draw,
                        minimum width = 2em},
  VertexStyle/.append style = { inner sep=5pt,
                                font = \Large\bfseries},
  EdgeStyle/.append style = {bend left} }
        \begin{tikzpicture}
        \tikzset{vertex/.style={circle,fill=blue!25,minimum size=12pt,inner sep=2pt}}
  \tikzset{every loop/.style={}}
    \node[vertex] (A) at (-4,0)  [shape=circle,draw=black,fill=gray] {};
        \node[vertex] (B) at (0,2)  [shape=circle,draw=black,fill=gray] {};
    \node[vertex] (C) at (0,-2)  [shape=circle,draw=black,fill=gray] {};
    \node[vertex] (D) at (4,0)  [shape=circle,draw=black,fill=gray] {};
     \node[vertex] (E) at (8,-2)  [shape=circle,draw=black,fill=gray] {};
      \node[vertex] (F) at (4,-4)  [shape=circle,draw=black,fill=gray] {};
    \path[loop/.style={looseness=30}]   (A) edge [loop left,  thick] node {} (A);
    \path[loop/.style={looseness=30}]   (B) edge [loop above,  thick] node {} (B);
    \path[loop/.style={min distance=10mm,in=210,out=240,looseness=30}]   (C) edge [loop,  thick] node {} (C);
     \path[loop/.style={min distance=10mm,in=30,out=60,looseness=30}]   (D) edge [loop,  thick] node {} (D);
     \path[loop/.style={looseness=30}]   (E) edge [loop right,  thick] node {} (E);
     \path[loop/.style={looseness=30}]   (F) edge [loop below,  thick] node {} (F);
    \path (B) edge [ thick] node[below]{$-$} (A);
    \path (A) edge [bend left,  thick] node[above]{$+$} (B);
     \path (C) edge [ thick] node[above]{$-$} (A);
    \path (A) edge [bend right,  thick] node[below]{$+$} (C);
      \path (B) edge [bend left,  thick] node[right]{$+$} (C);
    \path (C) edge [bend left,  thick] node[left]{$-$} (B);
   \path (C) edge [bend left=15, thick] node[above]{$-$} (D); 
   \path (D) edge [bend left=15, thick] node[below]{$+$} (C);
   \path (D) edge [thick] node[above]{$+$} (B);
  \draw [rounded corners, thick]  (D)  arc(0:180:4) (A);
\node[above] at (90:4) {$+$};
    \node[vertex] (A) at (-4,0)  [shape=circle,draw=black,fill=gray] {};
    \node[vertex] (D) at (4,0)  [shape=circle,draw=black,fill=gray] {};
        \path (D) edge [ thick] node[below]{$-$} (E);
    \path (E) edge [bend right,  thick] node[above]{$+$} (D);
    \path (F) edge [ thick] node[above]{$-$} (E);
    \path (E) edge [bend left,  thick] node[below]{$+$} (F);
    \path (F) edge [ thick] node[below]{$+$} (C);
       \path (D) edge [bend left,  thick] node[right]{$+$} (F);
    \path (F) edge [bend left,  thick] node[left]{$-$} (D);
      \draw [rounded corners, thick]  (C)  arc(180:360:4) (E);
      \node[below] at (4,-6) {$+$};
           \node[vertex] (E) at (8,-2)  [shape=circle,draw=black,fill=gray] {};
      \node[vertex] (F) at (4,-4)  [shape=circle,draw=black,fill=gray] {};
        \node[vertex] (C) at (0,-2)  [shape=circle,draw=black,fill=gray] {};
\end{tikzpicture}
        \caption{An OS-Koszul signed-graphic modular join which is not supersolvable}
        \label{fig:another_signed_graphic_matroid}
    \end{figure}
\end{Ex}
    
\section{Questions}\label{sec:questions}

We close by describing the relationships between various properties of hyperplane arrangements and highlighting some questions that remain open.  Let $\A = \{H_1,\ldots,H_n\}$ be a central essential hyperplane arrangement in $\CC^d$.  Each hyperplane is given as $H_i = V(\ell_i)$ with $\ell_i \in S = \CC[x_1,\ldots,x_d]$.
Let $\mathrm{Der}_{\CC}(S)$ be the $S$-module of $\CC$-derivations of $S$.
 The $S$-module $D(\A)$ of $\A$-derivations of $S$ is
\[
    D(\A) := \{\theta \in \mathrm{Der}_{\CC}(S) \mid \theta(\ell_i) \in (\ell_i) \text{ for } 1 \le i \le n\}.
\]
The arrangement $\A$ is \textbf{free} if $D(\A)$ is free as a graded $S$-module.  In this case, we can write $D(\A) = \bigoplus_{i = 1}^d S(-a_i)$. Terao \cite{Terao81} showed that the Poincar\'e polynomial factors as $\pi(\A,t) = \prod_{i = 1}^d (1 + a_i t)$.  He also showed \cite{Terao80a}*{Theorem 2.21} that fiber-type, and hence supersolvable, arrangements are free.  
These relationships are among those summarized in \Cref{quadracity:implications}.

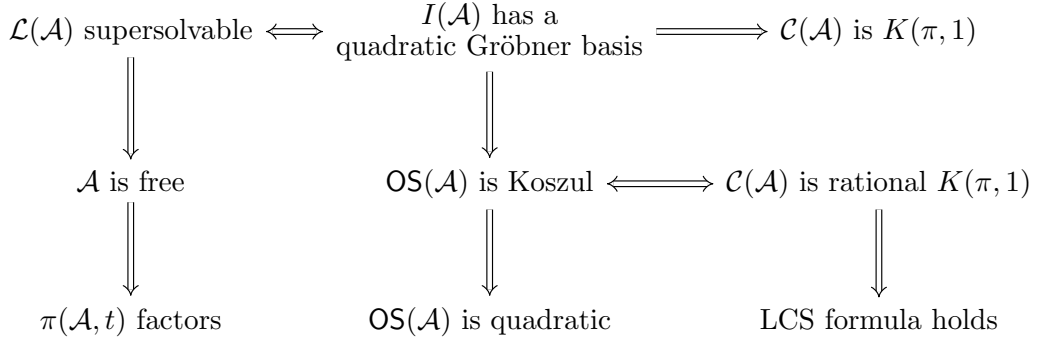
\begin{figure}[hbt!]
    \begin{center}
    \begin{tikzcd}[column sep = 2em, row sep = 3 em]
 \text{$\mathcal{L}(\A)$ supersolvable} \arrow[d,Rightarrow] \arrow[r,Leftrightarrow]& \substack{ \text{\normalsize $I(\mathcal{A})$ has a} \\ \text{\normalsize quadratic Gr\"obner basis}}\arrow[r,Rightarrow] \arrow[d,Rightarrow]   & \text{$\mathcal{C}(\A)$ is $K(\pi,1)$} \\
 \text{$\A$ is free}  \arrow[d,Rightarrow]
& \text{$\OS(\A)$ is Koszul} \arrow[d,Rightarrow]  \arrow[r,Leftrightarrow] & \text{$\mathcal{C}(\A)$ is rational $K(\pi,1)$}\arrow[d,Rightarrow] \\
 \text{$\pi(\A,t)$ factors} & \text{\normalsize $\OS(\mathcal{A})$ is quadratic}& \text{LCS formula holds}\\
    \end{tikzcd}
    \end{center}
    \caption{Some properties of complex hyperplane arrangements.}\label{fig1}
    \label{quadracity:implications}
\end{figure}

While our constructions show that the converse to the statement ``$\mathcal{L}(\A)$ supersolvable $\Rightarrow$ $\OS(\A)$ is Koszul'' is false, it leaves open many natural questions.  It is known that there are arrangements $\A$ such that $\mathcal{C}(\A)$ is $K(\pi,1)$ but not rational $K(\pi,1)$; see e.g. \cite{Falk95}*{Example 3.3}.  Whether the converse holds is an open question; cf. \cite{FR86}*{2.8}.  

\begin{Qst}
    If $\mathcal{C}(\A)$ is a rational $K(\pi,1)$ space, is $\mathcal{C}(\A)$ a $K(\pi,1)$ space?  
\end{Qst}
\noindent
In \cite{FR86}*{p.\ 113}, it was stated that OS-Koszul arrangements must be free, but this was based on incorrect assertion of Kohno.
As \Cref{qst:mainquestion} has been open since then, it is understandable that the question has not been asked again since. Now that \Cref{qst:mainquestion} has been settled in the negative, it is reasonable to reopen it. 

\begin{Qst}
    If $\mathcal{C}(\A)$ is a rational $K(\pi,1)$ space, must $\A$ be a free arrangement? 
\end{Qst}

\noindent With the notation of \Cref{thm:FP}, a related conjecture of Falk and Proudfoot \cite{falk_proudfoot}*{Conjecture 3.11} asks whether an arrangement $\A$ is free, assuming that $X$ is a modular flat and that $\A_X$ and $\A_\pi$ are free. As mentioned in the preamble to \Cref{cor:tsujie}, some special cases of our third construction yield \emph{modularly extended} arrangements, which were shown to be free by Tsujie \cite{Tsujie20}*{Theorem 1.8}.

On the other hand, \emph{all} our constructions preserve the property that the Poincar\'e polynomial factors into linear factors; this may be deduced from \cite{constructions}*{Corollary 7.4}; see also \cite{constructions}*{Corollary 7.8}.
Thus the following weaker question is also natural.

\begin{Qst}
    If $\OS(\A)$ is Koszul, must $\pi(\A,t)$ factor into linear factors?
\end{Qst}
\noindent
Parallel connections of supersolvable matroids yield $G$-quadratic Orlik--Solomon algebras by \Cref{prop:gquadratic}. 
We do not know if the same is true for our other mechanisms.
The natural next question is then the following.

\begin{Qst}
    Is there a central arrangement $\A$ whose Orlik--Solomon algebra is Koszul but not G-quadratic?
\end{Qst}
\noindent
Our methods likely do not construct all OS-Koszul matroids. For example, there is computational evidence that the Betsy Ross matroid \cite{LMMP24}*{Figure 2}, which is not supersolvable, has a Koszul Orlik--Solomon algebra. 
We were however unable to deduce this using the mechanisms described in this paper.
This leads us to ask the following.

\begin{Qst}
    Is the Betsy Ross matroid OS-Koszul?\footnote{That the Betsy Ross matroid might be OS-Koszul was first pointed out to the authors by Peeva.}
\end{Qst}

\noindent The Poincar\'e polynomial of the Betsy Ross matroid matches that of the arrangement in \Cref{fig:rank_3_non_ss}, which is OS-Koszul, so there can be no Hilbert function obstruction to Koszulness.  However, computing the syzygies of their defining ideals shows that their Orlik--Solomon algebras are not isomorphic.

\section*{Acknowledgements} This project began in December 2025 during the Arizona--New Mexico Symposium on Commutative Algebra and Its Interactions: Geometric Combinatorics organized by Michael Dipasquale, Louiza Fouli, and Jonathan Monta\~no.  The authors thank the organizers for their efforts.  This conference was supported by the National Science Foundation grants DMS–2401522 and DMS–2344588, the Department of Mathematics at New Mexico State University, and the School of Mathematics and Statistics at Arizona State University.

The authors are grateful to Matt Larson for helpful discussions and for valuable feedback on an earlier version of this manuscript.
The authors also thank G\"unter Ziegler for valuable correspondence regarding the correct statement of Theorem~\ref{thm:main}, and Vic Reiner for pointing out that graded Varchenko--Gel'fand rings behave similarly to Artinian Orlik--Terao algebras.   McCullough was partially supported by National Science Foundation grant  DMS--2401256.

Macaulay2 \cite{M2} computations were an invaluable tool for discovering and verifying the ideas that became the results of this paper.
No artificial intelligence tools or large language models were used in the creative or writing processes.

\bibstyle{amsalpha}
\bibliography{ref}
\end{document}